\documentclass[letterpaper,11pt,reqno]{amsart}
\usepackage[foot]{amsaddr}
\usepackage{latexsym,amssymb,amsmath,amsfonts,amsthm,mathrsfs,xcolor}
\usepackage{bbm}
\usepackage{dsfont}
\usepackage{enumerate,enumitem}
\newtheorem{theorem}{Theorem}[section]
\newtheorem{corollary}[theorem]{Corollary}
\newtheorem{lemma}[theorem]{Lemma}
\newtheorem{prop}[theorem]{Proposition}
\newtheorem{conj}{Conjecture}
\newtheorem{remark}[theorem]{Remark}
\newtheorem{definition}[theorem]{Definition}
\newtheorem{assumption}[theorem]{Assumption}

\allowdisplaybreaks

\def\beq{ \begin{equation} }
\def\eeq{ \end{equation} }
\def\mn{\medskip\noindent}

\def\ep{\epsilon}

\def\square{\vcenter{\vbox{\hrule height .4pt
  \hbox{\vrule width .4pt height 5pt \kern 5pt
        \vrule width .4pt} \hrule height .4pt}}}

\def\ZZ{\mathbb{Z}}
\def\FF{\mathcal{F}}

\def\AA{\mathcal{A}}

\def\WW{\mathcal{W}}

\def\KK{\mathcal{K}}

\def\ep{\varepsilon}

\definecolor{darkblue}{rgb}{0,0,0.6}

\def\Cay{\mathrm{Cay}}

\def\P{\mathbb{P}}
\def\E{\mathbb{E}}

\def\HH{\mathcal{H}}
\def\id{\mathrm{id}}
\def\1{\mathds{1}}
\def\0{\textbf{0}}
\def\Unif{\mathrm{Unif}}
\def\typ{\mathrm{\mathbf{typ}}}

\def\gcd{\mathrm{gcd}}

\def\Var{\mathrm{Var}}

\def\vp{\mathbf{p}}

\def\C{\mathbf{C}}
\def\BB{\mathcal{B}}

\def\err{\mathrm{err}}

\newcommand{\floor}[1]{\lfloor #1 \rfloor}

\usepackage{graphicx}
\graphicspath{%
    {converted_graphics/}
    {/}
}
\def\RR{\mathrm{R}}
\def\SS{\mathrm{S}}

\begin{document}

\title[Cutoff on dihedral graphs]{Cutoff for Random Walks on Dihedral Groups}
\author[Huang]{Xiangying Huang}
\email{zoehuang@unc.edu}
\author[Rao]{Renyu Rao}
\email{renyu@unc.edu}
\address{Department of Statistics and Operations Research, University of North Carolina at Chapel Hill, USA}

\begin{abstract}

We study the random walk on a finite dihedral group $G$ driven by the uniform measure on $k$ independently and uniformly chosen elements. We show that the walk exhibits cutoff with high probability throughout nearly the entire regime $1 \ll \log k \ll \log |G|$, and determine the precise cutoff time. Interestingly, this mixing time differs from the entropic time that characterizes cutoff behavior for random walks on Abelian groups. When $k \gg \log|G|$ and $\log k \ll \log|G|$, cutoff occurs with high probability on random Cayley graphs of virtually Abelian groups. The analysis develops techniques for obtaining sharper entropic estimates of an auxiliary process on high-dimensional lattices with dependent coordinates, which may also prove useful for related models in broader contexts.

\end{abstract}
\maketitle

\section{Introduction}

\subsection{Motivation and background}

The mixing behavior of random walks on finite groups is a classical problem in probability theory. Diaconis and Shahshahani \cite{diaconis1981generating} and Aldous and Diaconis \cite{aldous1986shuffling} studied card-shuffling models, which are random walks on the symmetric group. These studies provided some of the earliest rigorous proofs of the cutoff phenomenon, where the distance to stationarity drops abruptly from near its maximum to near zero over a short time window, known as the cutoff window.

Since then, the focus has broadened to other families of groups, with Abelian groups 
\cite{dou1996enumeration,dou1992studies,hermon2018supplementary,hermon2019further,hermon2021cutoff,hermon2021geometry,hough2017mixing} 
and nilpotent groups 
\cite{stong1995random,nestoridi2019super,peres2013mixing,ganguly2019upper,diaconis2021random,nestoridi2023random,diaconis1994moderate,hermon2024cutoff} 
being the most  extensively-studied examples. A central question motivating this line of research concerns the universality of cutoff, which is formalized as a conjecture by Aldous and Diaconis.
\begin{conj}[Aldous and Diaconis, 1985 \cite{AldousDiaconis1985techreport}]\label{universality_cutoff}
Let $G$ be a finite group and let $\mathcal{S} = \{ Z_i^{\pm1} : i \in [k]\}$ be a set where $Z_1,\dots,Z_k$ are i.i.d.\ uniform elements of $G$.  For $k$ satisfying $k \gg \log |G|$ and $\log k \ll \log |G|$,
 the random walk driven by the uniform measure over $\mathcal{S}$ exhibits cutoff with high probability.
\end{conj}

Indeed, the cutoff phenomenon is believed to occur universally in rapidly mixing, high-dimensional systems, such as random walks on groups driven by a sufficiently large number of generators. The verification of the Aldous-Diaconis universality conjecture has led to a substantial body of research, with considerable progress made \cite{dou1996enumeration, dou1992studies, hildebrand1994random, hildebrand2005survey, hough2017mixing, roichman1996random, wilson1997random, hermon2021cutoff, hermon2019cutoff, hermon2024cutoff}. Several of these notable contributions are reviewed in Section~\ref{universality_of_cutoff}.

Hermon and Olesker-Taylor~\cite{hermon2021cutoff} established that random walks on random Cayley graphs of Abelian groups exhibit cutoff with high probability in the extended (and necessary) regime $1 \ll \log k \ll \log|G|$. In this case, when the rank $r(G)$ of the Abelian group $G$ satisfies $r(G)\ll \log|G|$ and $k - r(G)\asymp k \gg 1$, the cutoff time is characterized as the moment when the entropy of a simple random walk on $\ZZ^k$ reaches $\log|G|$. 
Their work connects the mixing behavior of random walks on random Cayley graphs to the entropic evolution of an auxiliary process on the high-dimensional lattice $\ZZ^k$, where the cutoff phenomenon manifests as entropy concentration.

Subsequently, Hermon and Huang~\cite{hermon2024cutoff} extended the occurrence of cutoff  to random Cayley graphs of nilpotent groups throughout the regime $1 \ll \log k \ll \log|G|$, under some assumptions on the nilpotency class and rank. The mixing behavior in the nilpotent setting is largely governed by the walk's projection onto the Abelianization $G_{\mathrm{ab}} := G/[G,G]$, and the cutoff time is characterized by the moment when the entropy of a simple random walk on $\ZZ^k$ reaches $\log|G_{\mathrm{ab}}|$, provided $k$ is not too large. This indicates that the group structure does affect the entropy evolution of the original walk, an effect that was carefully captured through the analysis of a suitable auxiliary process in ~\cite{hermon2024cutoff}.

The reduction of the mixing behavior of random walks on nilpotent groups to their projections onto Abelian subgroups motivates the study of mixing behavior on groups that possess a structurally dominant Abelian component. Among such examples, the dihedral group provides a natural next case: it contains a large cyclic subgroup of index two, providing the simplest non-nilpotent setting where the Abelian structure remains significant.

The dihedral group of order $2n$ is generated by \textit{reflections} and \textit{rotations}, and is defined by
\beq\label{dihedral_group}
G = G^{(n)} = \langle r, s : r^n = \id, \; s^2 = \id, \; srs = r^{-1} \rangle, \qquad n \in \mathbb{N}.
\eeq
This group is solvable and virtually Abelian, but not nilpotent when $n$ is not a power of two. Nilpotent groups are often viewed as ``almost Abelian'' because they can be decomposed into successive Abelian quotients, whereas virtually Abelian groups are ``almost Abelian'' in a different sense, with a large Abelian subgroup of finite index. In this setting, the quotient-based analysis employed in \cite{hermon2024cutoff} no longer applies. This distinction is particularly evident in the study of dihedral groups, where the broader goal is to develop entropic techniques that extend to virtually Abelian groups.

For random walks on a dihedral group $G=G^{(n)}$ with prime $n$, driven by $k$ i.i.d. uniformly random generators, McCollum \cite{mccollum2011random, mccollum2013upper, mccollum2006random} showed that the mixing time is of order $n^{2/(k-1)}$ when $k$ is a fixed integer. When $k = \lfloor (\log |G|)^\alpha \rfloor$ with $\alpha > 1$, the analysis of Dou and Hildebrand~\cite{hildebrand1994random,dou1996enumeration} suggests that the cutoff occurs at time $\frac{\log |G|}{\log (k / \log |G|)}$. 
A unified perspective on these mixing times for different values of $k$ can be obtained by tracking the entropy evolution of an auxiliary process associated with the walk, which is defined in Definition~\ref{C_a}.
Heuristically, the proposed mixing time in~\eqref{cutoff_time} corresponds to the moment when the auxiliary process attains entropy $\log |G|$, while the cutoff phenomenon reflects the concentration of the random entropy around this point.

Our main result, Theorem~\ref{cutoff_iid}, establishes the occurrence of cutoff for simple random walks on dihedral groups driven by $k$ i.i.d.\ uniformly chosen generators, throughout almost the entire regime $1 \ll \log k \ll \log |G|$, at the mixing time $t_0(k,G)$ given in~\eqref{cutoff_time}. In particular, in the regime $k \gg \log|G|$ with $\log k \ll \log|G|$, cutoff is further established for the broader class of virtually Abelian groups at the same time $t_0(k,G)$, see Theorem~\ref{cutoff_virtually_abelian}.

The remainder of this section is organized as follows. In Section~\ref{main_result}, we state our main results precisely. Section~\ref{entropic_framework} outlines the methodological framework used in the proofs, and Section~\ref{universality_of_cutoff} reviews prior work on the universality of cutoff for random walks on groups.

\subsection{Statement of results}\label{main_result}

We now introduce the precise setting and notation, and state our main results. Random walks on groups are naturally interpreted as random walks on \emph{Cayley graphs} generated by a given set of generators $\mathcal{S}\subseteq G$, defined as follows. A set $\mathcal{S} \subseteq G$ is called \emph{symmetric} if $g^{-1} \in \mathcal{S}$ whenever $g \in \mathcal{S}$, and is called a set of generators if $\mathcal{S}$ generates $G$.

\begin{definition}[Undirected Cayley graph]\label{Cayley_graph}
Let $G$ be a finite group and $\mathcal{S} = \{ z_i^{\pm 1} : i \in [k] \} \subseteq G$ a symmetric set of generators. Let $\Cay(G, \mathcal{S})$ denote the (right) Cayley graph on $G$ generated by $\mathcal{S}$, where the vertex set is
$
\mathbf{V} = \{ g : g \in G \}
$
and the edge set is
$
\mathbf{E} = \{ \{g, gz\} : g \in G, z \in \mathcal{S} \}.
$
\end{definition}
\begin{remark}
The generating set $\mathcal{S}$ is a multiset, i.e., we allow parallel edges and self-loops (if $\id \in \mathcal{S}$), so that the Cayley graph $\Cay(G,\mathcal{S})$ is regular of degree $2k$.
\end{remark}

We use standard notation and definitions for mixing times and cutoff. Let $G$ denote a finite group, and $P$ denote the transition kernel of the simple random walk driven by the set of generators $\mathcal{S}\subseteq G$, i.e., for $x,y\in G$,
$$P(x,y)=\frac{1}{|\mathcal{S}|}\cdot \1\{x^{-1}y\in \mathcal{S}\}.$$
The stationary distribution $\pi$ is uniform over $G$. For $t \in \mathbb{R}_+$, the transition kernel of a continuous-time simple random walk $X=(X(t))_{t\geq 0}$ with rate 1 is given by
$
P_t := \sum_{k=0}^\infty \frac{t^k}{k!} P^k e^{-t},
$
with $P$ defined as above. The walk $X=(X(t))_{t\geq 0}$ is also referred to as the random walk on $\Cay(G,\mathcal{S})$.

The total variation distance to stationarity at time $t$ is defined as
\[
d_{\mathrm{TV}}(t):=\max_{x\in G}\|P_t(x, \cdot) - \pi\|_{\mathrm{TV}} := \frac{1}{2} \max_{x\in G}\sum_{y\in G} |P_t(x,y) - \pi(y)|,
\]
and the $\varepsilon$-total variation mixing time is given by
\[
t_{\mathrm{mix}}(\varepsilon) := \min\left\{t : d_{\mathrm{TV}}(t)\leq \varepsilon \right\}.
\]

The rate-1 simple random walk $X=(X(t))_{t\geq 0}$ on the Cayley graph $\Cay(G,\mathcal{S})$ is said to exhibit cutoff at time $t_0=t_0(G,\mathcal{S})$ if, for any $\ep>0$, 
$$d_{\mathrm{TV}}((1-\ep)t_0)=1-o(1) \quad \text{ and }\quad d_{\mathrm{TV}}((1+\ep)t_0)=o(1)$$ 
when $|G|$ is growing to infinity. In what follows, we take $G=G^{(n)}$ to be a sequence of finite groups with diverging size.

\medskip

Let $\lambda>0$ be a constant. We define the proposed mixing time 
\begin{equation}\label{cutoff_time}
t_0(k,G) =
\begin{cases}
\dfrac{k}{2\pi e} \, |G|^{2/(k-1)} & \text{when } k \ll \log |G|,\\[2mm]
t_{\mathrm{ent}}(k,|G|) & \text{when } k \eqsim \lambda \log |G|,\\[1mm]
\dfrac{\log |G|}{\log (k / \log |G|)} & \text{when } k \gg \log |G|,
\end{cases}
\end{equation}
where $t_{\mathrm{ent}}(k,|G|)$  is the time at which the  entropy of a rate-1 simple random walk on $\ZZ^k$ reaches $\log|G|$, see Proposition  \ref{asymp_entropic_time} for its known sharp asymptotics.

 \begin{theorem}\label{cutoff_iid}
 Let $G = G^{(n)}$ be a sequence of finite dihedral groups as in \eqref{dihedral_group} with prime $n$, and let $k = k_n$ diverge with $n$. Let $\mathcal{S} = \{Z_i^{\pm 1} : i \in [k]\}$, where $Z_1, \dots, Z_k$ are i.i.d.\ uniform elements of $G$. Assume that $1 \ll \log k \ll \log |G|$.
 Within this range, if either
\[
\text{(i)} \; k^2 \ll |G|^{2/k} \quad \text{or} \quad 
\text{(ii)} \; k \gg |G|^{2/k} (\log k)^2,
\]
then the random walk on $\Cay(G, \mathcal{S})$ exhibits cutoff with high probability at the time $t_0(k,G)$ defined in \eqref{cutoff_time}.

 \end{theorem}

A simpler (though not sharpest) interpretation of Theorem~\ref{cutoff_iid} is that cutoff occurs with high probability if either
\[
k \leq \frac{\log |G|}{\log \log |G|} \qquad \text{or} \qquad 
k \geq \frac{2(1+\varepsilon)\log |G|}{\log \log |G|} \quad \text{for any } \ep > 0.
\]
We believe that the gap in the middle is a technical artifact of our entropic estimates, and that cutoff should occur throughout the entire regime $1 \ll \log k \ll \log |G|$.

We also comment on the distinction between 
$t_0(k,G)=\frac{k}{2\pi e}\,|G|^{2/(k-1)}$ 
and 
$t_{\mathrm{ent}}(k,G)\asymp \frac{k}{2\pi e}\,|G|^{2/k}$ 
in the regime $k\ll \log |G|$. 
In particular, the discrepancy in the exponents is asymptotically significant when $|G|^{2/k^2}\gg 1$, but becomes negligible once $k$ is sufficiently large. 
From an entropic viewpoint, the exponent $\frac{2}{k-1}$ arises because the auxiliary process effectively has dimension $k-1$, which reflects the algebraic structure of the dihedral groups. This point will be clarified further in Section~\ref{sec: asymptotics of entropy}.

\medskip

In the regime $k\gg \log|G|$ with $\log k\ll \log|G|$, cutoff can be shown for the general class of virtually Abelian groups at time $t_0(k,G)$.

 \begin{theorem}\label{cutoff_virtually_abelian}
Let $G = G^{(n)}$ be a sequence of finite virtually Abelian groups, and let $H = H^{(n)} \trianglelefteq G^{(n)}$ be a sequence of normal Abelian subgroups such that $\sup_n |G^{(n)}|/|H^{(n)}|<\infty$.
Let $k = k_n$ and 
$
\mathcal{S} = \{ Z_i^{\pm1} : i \in [k] \},
$
where $Z_1, \dots, Z_k$ are i.i.d.\ uniform elements of $G$.  
If $k \gg \log |G|$ and $\log k \ll \log |G|$, then the simple random walk on $\Cay(G, \mathcal{S})$ exhibits cutoff with high probability at time
$$
t_0(k,G) = \frac{\log |G|}{\log \bigl( k / \log |G| \bigr)}.
$$
\end{theorem}

 We note that virtually Abelian groups are a special case of the groups considered in Dou and Hildebrand \cite[Assumption 1]{dou1996enumeration}, where they prove a sharp lower bound on the mixing times when $k = \lfloor (\log |G|)^\alpha \rfloor$ with $\alpha > 1$; see \cite[Theorem 4]{dou1996enumeration}. In Theorem \ref{cutoff_virtually_abelian}, we extend the result to the regime where $k \gg \log|G|$ and $\log k \ll \log|G|$.

\subsection{Overview and Methodology}\label{entropic_framework}
Studying the mixing behavior through the entropy of an auxiliary process is commonly referred to as the \emph{entropic methodology}. This approach has proven effective for random walks driven by random generators in several prior works on Abelian and nilpotent groups \cite{hermon2019cutoff, hermon2021cutoff, hermon2024cutoff}, though a key challenge is that the auxiliary processes depend on the group's commutativity and can be quite intricate to analyze.

To illustrate the main ideas of the entropic methodology, we first describe its application to Abelian groups, following the approach in \cite{hermon2021cutoff}. For a rate-1 random walk $X=(X(t))_{t\geq 0}$ on an Abelian group $G$ with a symmetric generating set $\mathcal{S}=\{ s_i^{\pm 1}: i \in [k]\}$, one introduces an auxiliary process 
\[
W=W(t) = \big(W_1(t),\dots, W_k(t)\big), \quad \text{for }t\geq0,
\] 
where $W_i(t)$ records the number of times generator $s_i$ has been applied minus the number of times its inverse $s_i^{-1}$ has been applied in the walk up to time $t$. It is straightforward to observe that $W$ is a rate-1 simple random walk on $\ZZ^k$ and that 
$$X(t)=\sum_{i\in[k]}W_i(t)s_i.$$

Let $t_\mathrm{ent}(k,G)$ denote the time at which the entropy of $W$ reaches $\log |G|$. Intuitively, before this time, $W$ is largely concentrated on a set of size $o(|G|)$, so the associated walk $X(t)$ is also confined to a set of size $o(|G|)$ and cannot be mixed yet. After $t_\mathrm{ent}(k,G)$, one can show that with high probability $W(t)$ lies in the set
\[
\WW(t)=\{ w\in \ZZ^k : \P(W(t)=w)\leq e^{-\omega}|G|^{-1} \}
\]
for some slowly diverging $\omega\gg 1$, see, e.g., Proposition 2.3 of \cite{hermon2019cutoff}. The entropic framework then allows one to bound the mixing time via the following observations.

\begin{lemma}[Lemma 2.6 of \cite{hermon2021cutoff}]\label{TV_entropic}
Let $\P_\mathcal{S}$ denote the law of the random walk $X=(X(t))_{t\geq 0}$ on an Abelian group $G$ driven by $\mathcal{S}$. For any $t \geq 0$ and any set $\mathcal{W}\subseteq \ZZ^k$,
\begin{align*}
 \left\| \P_{\mathcal{S}}(X(t)=\cdot)-\pi \right\|_{\mathrm{TV}}
&\leq \left\| \P_{\mathcal{S}}(X(t)= \cdot | W(t)\in \mathcal{W})-\pi \right\|_{\mathrm{TV}}
+ \P(W(t)\notin \mathcal{W}), \\
4\, \left\|\P_{\mathcal{S}}(X(t)= \cdot | W(t)\in \mathcal{W})-\pi \right\|^2_{\mathrm{TV}}
&\leq |G|\cdot \P_{\mathcal{S}}\left(X(t)=X'_t| W(t),W'(t)\in \mathcal{W}\right)-1,
\end{align*}
where $X'_t$ and $W'(t)$ are independent copies of $X(t)$ and $W(t)$, respectively.
\end{lemma}

Lemma~\ref{TV_entropic} reduces the total variation distance analysis to a hitting probability, which can be further decomposed as
\begin{align*}
\P_{\mathcal{S}}\left(X(t)=X'_t| W(t),W'(t)\in \mathcal{W}\right)
\leq& \, \P_{\mathcal{S}}(W(t)=W'(t)| W(t),W'(t)\in \mathcal{W}) \\
&+ \P_{\mathcal{S}}(X(t)=X'(t),\, W(t)\neq W'(t)| W(t),W'(t)\in \mathcal{W}),
\end{align*}
where the first term is controlled via the entropic analysis of $W$. This captures the main idea behind the entropic methodology. Consequently, the entropic time $t_\mathrm{ent}(k,G)$ determines the mixing (and cutoff) time of the walk $X$ for Abelian groups satisfying appropriate regularity conditions, providing explicit asymptotics for the mixing times in these cases. We present the asymptotics of $t_\mathrm{ent}(k,G)$ below, following Proposition A.2 of \cite{hermon2018supplementary}.

\begin{prop}[Proposition A.2 \cite{hermon2018supplementary}]\label{asymp_entropic_time}
Let $t_{\mathrm{ent}}(k,N)$ denote the time at which the entropy of a rate-1 simple random walk on $\mathbb{Z}^k$ reaches $\log N$. We refer to $t_{\mathrm{ent}} := t_{\mathrm{ent}}(k, |G|)$ as the \emph{entropic time}, which satisfies the following asymptotics:
\[
t_{\mathrm{ent}}(k,|G|) \eqsim
\begin{cases}
\dfrac{k}{2\pi e}  |G|^{2/k} & \text{when } k \ll \log |G|,\\[1mm]
k f(\lambda) & \text{when } k \eqsim \lambda \log |G|,\\[1mm]
\dfrac{\log |G|}{\log (k / \log |G|)} & \text{when } k \gg \log |G|,
\end{cases}
\]
where $f: \mathbb{R}_+ \to \mathbb{R}_+$ is a continuous function.
\end{prop}

\medskip
Additional challenges arise when the underlying group is non-Abelian. For nilpotent groups, analyzing the auxiliary process $W(t)$ alone is insufficient, and one must consider an additional auxiliary process that captures the pairwise relative orders, see \cite{hermon2024cutoff}. In our study of dihedral groups, the auxiliary process $\C(t)$ identified in \eqref{C_a} is not a simple random walk on a high-dimensional lattice; rather, it exhibits weak dependence across coordinates, which considerably complicates the analysis of its entropic properties.

The entropic analysis of the auxiliary process $\C(t)$ in Section~\ref{sec:entropy_auxiliary_process} forms the technical core of this paper and underlies our characterization of the mixing time in \eqref{cutoff_time}.
At the heart of our work is the analysis of the entropy evolution of a continuous-time process $(Y_t)_{t\in\mathbb{R}_+}$ on $\ZZ^d$, which jumps at rate $1$. Its discrete-time skeleton at the jump times, $(Y_n)_{n \ge 0}$, satisfies
\[
Y_n = -Y_{n-1} + \delta_n = \sum_{i=1}^n (-1)^{n-i} \delta_i \qquad \text{for $n\ge 1$},
\]
with $Y_0 = 0$ and $\{\delta_n\}_{n \ge 1}$ i.i.d.\ uniform on $\{e_i : i \in [d]\}$. Here, the dimension $d$ is taken to be of the same order as $k$, and $n$ corresponds to the number of jumps up to a given time $t$.
It is straightforward to see that $Y_n$ has the law of the difference of two multinomial random variables on $\mathbb{Z}^d$, and is therefore dependent across coordinates. 

In Proposition~\ref{entropy_approximation}, we develop a comparison argument with simple random walk on $\mathbb{Z}^d$ to derive sharp asymptotics of the entropy of $Y_n$, which, to the best of our knowledge, is a novel result and can be applied more generally to processes that ``resemble'' directed or undirected random walks while exhibiting some dependence across coordinates. Different methods are employed to establish entropic concentration of $Y_t$ across various regimes of $d$, which connect to other questions of interest, such as the entropic concentration of multivariate normal random variables \cite{BobkovMadiman2011} and  the Chung--Diaconis--Graham process \cite{eberhard2021mixing}.

\subsection{Literature on the universality of cutoff}\label{universality_of_cutoff} 
The Aldous--Diaconis universality conjecture (Conjecture~\ref{universality_cutoff}) was first verified in the setting of Abelian groups. 
Cutoff was initially established for $k \ge \lfloor (\log |G|)^\alpha \rfloor$ with $\alpha > 1$ by Dou and Hildebrand~\cite{dou1996enumeration,dou1992studies}, 
and was later extended to the full regime $1 \ll \log k \ll \log |G|$ by Hermon and Olesker-Taylor~\cite{hermon2021cutoff}, which is both necessary and sufficient for cutoff to occur. 
For a detailed account of the progress on Abelian groups, we refer the interested reader to~\cite[\S1.3.1]{hermon2021cutoff}.

Beyond Abelian groups, Hermon and Olesker-Taylor \cite{hermon2019cutoff} studied random walks on the unit upper-triangular matrix groups, a canonical example of nilpotent groups, and established cutoff in the same extended regime $1 \ll \log k \ll \log |G|$ with precise mixing times. Their conclusion was subsequently generalized by Hermon and Huang \cite{hermon2024cutoff}, where  the universality conjecture is extended to general nilpotent groups for $1 \ll \log k \ll \log |G|$, under certain regularity assumptions on the nilpotency class and rank. In particular, \cite{hermon2024cutoff} provides a characterization of the mixing time in terms of the projected walk onto the abelianization, showing that the mixing behavior on the nilpotent group is largely determined by its projected walk onto the abelianization.

Dou and Hildebrand \cite{dou1996enumeration} outlined a proof of cutoff for a certain class of groups under the choice $k = \lfloor (\log |G|)^\alpha \rfloor$ with $\alpha > 1$; see Assumption~1 and Theorem~4 of \cite{dou1996enumeration}. This class includes virtually Abelian groups and thus dihedral groups. 
In working through the argument, we observed that several steps were only sketched, and these require further development to yield a complete proof. Hence, in Section~\ref{sec:upper_regime}, we present a complete proof of Theorem~\ref{cutoff_virtually_abelian} and refine the analysis to cover the entire regime $k \gg \log |G|$.

\medskip
The setting in which the generating set is fixed has also received considerable attention, see e.g. \cite{stong1995random, pak2000two,peres2013mixing,ganguly2019upper, nestoridi2020random, diaconis1994moderate, diaconis2021random,hough2017mixing}, although cutoff in this case remains far less understood. Hough~\cite{hough2017mixing} proved general bounds on the mixing time on cyclic groups $\ZZ_p$ through an elegant connection with lattice theory. For nilpotent groups, various results indicate that the mixing time is largely governed by the Abelianization. Diaconis and Hough~\cite{diaconis2021random} developed a general framework for proving central limit theorems for random walks on unipotent matrix groups driven by a probability measure $\mu$, applicable to broad choices of generators. For the $d\times d$ unit upper-triangular group $\mathbb{U}_d(p)$ over $\ZZ_p$, they showed that the coordinate on the $k$-th diagonal mixes in order $p^{2/k}$ steps, highlighting the dominant role of the Abelianization (the first super-diagonal) in the overall mixing. Nestoridi and Sly~\cite{nestoridi2020random} investigated the mixing behavior of the random walk on $\mathbb{U}_d(p)$ with the canonical generating set ${I \pm E_{i,i+1} : i \in [d-1]}$. They showed that the mixing time is bounded by $O(p^2 d \log d + d^2 p^{o(1)})$, where the first term in the bound corresponds to the mixing behavior of the projected walk on the Abelianization. For nilpotent groups of bounded nilpotency class and rank, Hermon and Huang~\cite{hermon2024cutoff} showed that the mixing behavior of the random walk is essentially determined by its projection onto the Abelianization, and that cutoff occurs precisely when the projected walk exhibits cutoff.

\medskip
We begin in Section~\ref{sec:preliminaries} with the necessary preliminaries, and proceed in Section~\ref{sec:entropy_auxiliary_process} to develop the entropic analysis that forms the core of the proof. The mixing-time proof is divided into lower and upper bounds, and further into regimes of $k$ that require distinct techniques; these are addressed in Sections~\ref{sec:lower_regime} and~\ref{sec:upper_regime}.

\section{Preliminaries}\label{sec:preliminaries}

We consider the dihedral group $G=G^{(n)}$ of size $2n$ with prime $n$, which is characterized by
$$G=\langle r,s :r^n=\id, s^2=\id, srs= r^{-1}\rangle,$$
where $r$ can be interpreted as the counterclockwise rotation by $2\pi /n$ and $s$ as the reflection across the horizontal axis. An element $g\in G$ is said to be a \textit{reflection} if $g^2=\id$. That is, the dihedral group $G$ can be viewed as a group  of rotations and reflections. 

Note that $G$ is not nilpotent when $n$ is an odd prime, 
since a dihedral group of order $2n$ is nilpotent if and only if $n$ is a power of $2$. 
A group is said to be virtually Abelian if it has an Abelian subgroup of finite index (independent of the size of the group). 
Hence $G$ is virtually Abelian with an Abelian subgroup $\ZZ_n$ of index $2$.

Let $\mathcal{S} = \{ Z^{\pm 1}_a : a \in [k] \} \subseteq G$ be a set of generators, where 
$Z_1, \dots, Z_k \overset{\mathrm{iid}}{\sim} \mathrm{Unif}(G)$. Note that $G$ is generated by $\mathcal{S}$ provided that $\mathcal{S}$ contains at least one rotation and one reflection, which occurs with high probability.
We will refer to
\[
\mathcal{S} \cap \{ r^i : i \in \mathbb{Z}_n \} 
\quad \text{and} \quad 
\mathcal{S} \cap \{ s r^i : i \in \mathbb{Z}_n \}
\]
as the subsets of generators consisting of rotations and reflections, respectively. Throughout this paper we will write 
$$k_\SS:=| \mathcal{S} \cap \{ r^i : i \in \ZZ_n \}| \quad \text{ and }\quad k_\RR:=| \mathcal{S} \cap \{ s r^i : i \in \ZZ_n \}|=k-k_\SS.$$
As $\mathcal{S}$ is a set of $k$ i.i.d. uniform generators, we have $k_\SS \sim \mathrm{Binomial}(k, 1/2)$. 
We will first reveal the values of $k_\SS$ and $k_\RR$ and treat them as known constants throughout our proofs. 

The following assumption holds with high probability and will be imposed throughout our discussion.

\begin{assumption}\label{assump_k_S}
We assume that
\[
\rho_\SS :=k_\SS/k \in [1/4, 3/4] 
\quad \text{and} \quad
\rho_\RR :=k_\RR/k \in [1/4, 3/4].
\]
\end{assumption}
The purpose of this assumption is simply to guarantee that both $k_\SS$ and $k_\RR$ are of the same order as $k$, a condition that we will take as default in what follows.

\mn
\textbf{Notation.} Throughout the paper, we use standard asymptotic notation: ``$\ll$'' or ``$o(\cdot)$'' means ``of smaller order''; ``$\lesssim$'' or ``$O(\cdot)$'' means ``of order at most''; ``$\asymp$'' means ``of the same order''; ``$\eqsim$'' means ``asymptotically equivalent,'' i.e., $f(x) \eqsim g(x)$ if $\lim_{x\to+\infty} f(x)/g(x) = 1$. 

For $n \in \mathbb{N}$, we write $[n] := \{1, \dots, n\}$.  
With slight abuse of notation, for a vector $x \in \mathbb{R}^d$, we let $[x]$ denote the nearest integer point in $\mathbb{Z}^d$.

\subsection{Representation of the random walk}

Writing $\mathcal{S}=\{Z_a^{\pm 1} : a \in [k]\}$, and conditioning on the values of $(k_\SS, k_\RR)$, each generator can be represented as follows:
\beq\label{gen_rep}
Z_a=
\begin{cases}
sr^{U_a} & \text{ for }1\leq a\leq k_\SS\\
r^{U_a}& \text{ for } k_\SS<a\leq k,
\end{cases}
\eeq
where $\{U_a\}_{a\in [k]}$ are  i.i.d. uniform over $\ZZ_n$.

Let $X := X(t)$ denote the continuous-time rate-1 random walk on the Cayley graph $\Cay(G, \mathcal{S})$, where, at rate 1, a generator is chosen uniformly from $\mathcal{S}$ and applied to the right of the current position. Let $N := N(t)$ denote the total number of steps taken by $X$ up to time $t$. For $1 \le i \le N$, we may write the $i$-th step of $X$ as $(Z_{\sigma_i})^{\eta_i}$,
where
\[
\sigma_i \;\overset{\mathrm{iid}}{\sim}\; \Unif([k]) 
\quad \text{and} \quad 
\eta_i \;\overset{\mathrm{iid}}{\sim}\; \Unif\{\pm 1\},
\qquad \text{ for }i\in\mathbb{N}.
\]
Consequently, the random walk admits the representation
\begin{equation}\label{seq_X}
    X \;=\; \prod_{i=1}^{N} (Z_{\sigma_i})^{\eta_i}.
\end{equation}
In what follows, we often refer to \eqref{seq_X} as the sequence $X$, to emphasize its interpretation as an ordered product of generators.

To obtain a simplified representation of $X$, we rearrange the sequence in \eqref{seq_X} so that all reflections appear to the left and all rotations to the right. The goal is to express $X$ as  
\begin{equation}\label{reflection&rotation}
    X \;=\; (X|_\SS)(X|_\RR),
\end{equation}
where $X|_\SS$ denotes the part of the rearranged sequence consisting of reflections, and $X|_\RR$ that consisting of rotations.

To this end, we traverse the sequence in \eqref{seq_X} from left to right, exchanging only the positions of a reflection and a rotation. That is, for each rotation $r$ and reflection $s$ in the sequence, we perform the rearrangement
\begin{equation}\label{rearrangement}
    rs = sr^{-1},
\end{equation}
which shifts the reflection $s$ to the left. For instance, the sequence $r_1 s_2 r_3 s_4$ is rearranged as follows:
\[
r_1 s_2 r_3 s_4 
\;\;\longrightarrow\;\; 
s_2 r_1^{-1} r_3 s_4 
\;\;\longrightarrow\;\; 
s_2 s_4 \bigl( r_1^{-1} r_3 \bigr)^{-1} 
\;=\; s_2 s_4 r_1 r_3^{-1},
\]
where the last equality follows because rotations commute.

Observe that, since the relative order of reflections is never changed, $X|_\SS$ coincides precisely with the subsequence of $X$ consisting of generators in $\{Z_a^{\pm 1} : a \in [k_\SS]\}$, i.e., the reflections (see \eqref{gen_rep}). Hence, letting $N_\SS=N_\SS(t)$ denote the total number of arrivals of generators in $\{Z_a^{\pm 1}: a\in [k_\SS]\}$ by time $t$, we can write 
\beq\label{X_S}
X|_\SS=\prod_{i=1}^{N_\SS} (Z_{\tilde\sigma_i})^{\tilde\eta_i}=\prod_{i=1}^{N_\SS}Z_{\tilde\sigma_i}
\eeq
where $(\tilde\sigma_i,\tilde\eta_i)\in [k_\SS]\times \{\pm 1\}$ denotes the index and sign of the $i$-th generator in the subsequence $X|_{\SS}$, and the last equality follows from the fact that a reflection is equal to its inverse. 

To obtain an expression for $X|_\RR$, we also need to account for the effect of the rearrangement on the rotations. 
A precise derivation is given in Proposition~\ref{simplified_X}, which represents the walk $X$ using an auxiliary process.

Recall that  $N(t)$ denotes the total number of generator arrivals by time $t$, and $N_\SS(t)$ denote the number of arrivals of generators in $\{Z_a^{\pm 1} : a \in [k_\SS]\}$ by time $t$.

\begin{prop}\label{simplified_X}
Let $\{(\tilde\sigma_i, \tilde\eta_i)\}_{i \ge 1}$ be as in \eqref{X_S}. 
Define the auxiliary process $\C(t) = (C_a(t))_{a \in [k]}$ by
\beq\label{C_a}
C_a(t) =
\begin{cases}
(-1)^{N_\SS(t) \bmod 2} \sum_{i=1}^{N_\SS(t)} (-1)^i \1\{\tilde\sigma_i = a\}, & 1\leq a \leq k_\SS,\\[1mm]
\displaystyle \sum_{i=1}^{N(t)} \1\{\sigma_i = a\} \eta_i \, (-1)^{\sum_{j>i} \1\{\sigma_j \in [k_\SS]\}}, & k_\SS < a \le k.
\end{cases}
\eeq
Then the random walk $X(t)$ can be expressed as
\[
X(t) = s^{N_\SS(t) \bmod 2} \, r^{\sum_{a \in [k]} C_a(t) U_a}.
\]
\end{prop}

\begin{proof}
For simplicity of notation, we drop the index on time. Following \eqref{reflection&rotation}, the goal is to identify the explicit expression of $X|_\SS$ and $X|_\RR$ obtained from rearranging. By \eqref{X_S} we can express $X|_\SS$ as
\beq\label{seq_X_refl}
X|_{\SS}=\prod_{i=1}^{N_\SS} Z_{\tilde\sigma_i}=\prod_{i=1}^{N_\SS}  (sr^{U_{\tilde\sigma_i}}),
\eeq
where $\tilde\sigma_i \in [k_\SS]$ denotes the index  of the $i$-th generator in $X|_\SS$, counted from the left.
Upon noting that for $a,b\in \ZZ_n$,
$$(sr^a)(sr^b)=r^{b-a},$$
we simplify the expression of \eqref{seq_X_refl} as
$$
X|_{\SS}=s^{(N_\SS \mod 2)} r^{ (-1)^{(N_\SS\mod 2)} \sum_{i=1}^{N_\SS} (-1)^{i} U_{\tilde\sigma_i}}=s^{(N_\SS \mod 2)}r^{\sum_{a\in[k_\SS]}C_aU_a}.$$

It remains to understand $X|_\RR$. From \eqref{rearrangement}, we see that the sign of a rotation flips each time it is exchanged with a reflection. For a rotation $(Z_{\sigma_i})^{\eta_i}$ in the sequence \eqref{seq_X}, after the rearrangement its sign becomes
\[
\eta_i (-1)^{\sum_{j>i} \1\{\sigma_j \in [k_\SS]\}}.
\]
More precisely, this effect is captured in the definition of $(C_a)_{k_\SS < a \le k}$:
\[
C_a = \sum_{i=1}^N \1\{\sigma_i = a\} \eta_i \, (-1)^{\sum_{j>i} \1\{\sigma_j \in [k_\SS]\}}, \quad\text{ for } k_\SS < a \le k.
\]
Since rotations commute with each other, it follows that
\[
X|_\RR = \prod_{k_\SS < a \le k} Z_a^{C_a} = r^{\sum_{k_\SS < a \le k} C_a U_a}.
\]
Combining $X|_\SS$ and $X|_\RR$ in \eqref{reflection&rotation} completes the proof.
\end{proof}

As indicated by Proposition~\ref{simplified_X}, the mixing behavior of the walk $X(t)$ is largely governed by its auxiliary process $\C(t)$, whose distribution plays a central role in analyzing the mixing of $X(t)$. We now proceed to characterize the law of $\C(t)$. Write $\C(t) = (\C_\SS(t), \C_\RR(t))$, where 
\[
\C_\SS(t) = (C_a(t))_{a \in [k_\SS]} \quad \text{and} \quad 
\C_\RR(t) = (C_a(t))_{k_\SS < a \le k}.
\] 
To describe the law of $\C(t)$, we introduce the following filtrations.

\begin{definition}[Filtrations]\label{filtration}
Let $\rho_\SS := k_\SS/k$ and $\rho_\RR := k_\RR/k$, which are treated as known constants.

\begin{enumerate}
\item Let $P_\SS = (P_\SS(t))_{t \ge 0}$ denote a Poisson process of rate $\rho_\SS$ that records the arrival times of generators that are reflections, without distinguishing their identities. Let 
\[
\FF_\SS(t) = \sigma\bigl( (P_\SS(s))_{0 \le s \le t} \bigr)
\] 
be the natural filtration of $P_\SS$ up to time $t$. The processes $P_\RR$ and filtrations $\FF_\RR(t)$ are defined analogously for rotations.

\item Write  $P=(P_\SS,P_\RR)$ and let $P'$ be an independent copy of $P$. We denote by $N$ (resp. $N'$) the number of arrivals in $P$ (resp. $P'$). Let $\HH$ be the $\sigma$-field generated by $P, P',  (\sigma_i, \eta_i)_ { i\in [N+N']}$, i.e., $\HH$ contains information on the sequences $X, X'$, other than the identities of $(U_a)_{a\in[k]}$ in \eqref{gen_rep}. \\
\end{enumerate}
\end{definition}

For simplicity, we will often drop the explicit dependence on $t$ and write $\C_\SS = \C_\SS(t)$, $\FF_\SS = \FF_\SS(t)$, etc. We next introduce the necessary notation and characterize the law of $\C$. 

\begin{definition}[Multinomial distribution]\label{multinomial}
Fix $d\geq 2$ and $m\in\mathbb{N}$. Given $\vp\in (0,1)^{d}$  with $\| \vp \|_1:=\sum_{i=1}^{d} p_i =1$, we will denote by 
$\mathrm{Multi}_{d}(m,\vp)$ the multinomial distribution in $\ZZ^{d}$, with probability mass function defined by 
$$p_m(x)=\frac{m!}{ \prod_{i=1}^{d} x_i !}\prod_{i=1}^{d} p_i^{x_i} \quad\quad \text{ for }x \in \{y\in \ZZ_+^d: \sum_{i=1}^d y_i=m\},$$
where $\ZZ^+_d=\{x\in \ZZ^d: x_i\geq 0 \text{ for all }i\in [d]\}$.

\end{definition}

\begin{prop}[Conditional law of the auxiliary process]\label{auxiliary_law}
Let $\FF_\SS$ be as defined in Definition~\ref{filtration}. Then we have
\begin{enumerate}[label=(\roman*)]
\item  The process $\C_\RR$ is independent of $\C_\SS$ and $\FF_\SS$, and evolves as a continuous-time simple random walk on $\mathbb{Z}^{k_\RR}$ with rate $\rho_\RR = k_\RR/k$;
%
\item Let $\vp = \frac{1}{k_\SS} (1,1,\dots,1) \in \mathbb{R}^{k_\SS}$, and let 
\[
\C_\SS^{\pm} \overset{\mathrm{iid}}{\sim} \mathrm{Multi}_{k_\SS}\bigl(\lfloor N_\SS/2 \rfloor, \vp\bigr), \qquad 
\C_{\err} \sim \mathrm{Multi}_{k_\SS}\bigl(\1\{N_\SS \text{ is odd}\}, \vp\bigr)
\] 
be conditionally independent multinomial random variables given $N_\SS$. The conditional law of $\C_\SS$ given $\FF_\SS$ can be written as
\[
\C_\SS \,|\, \FF_\SS = \C_\SS \,|\, N_\SS \overset{d}{=} \C_\SS^+ - \C_\SS^- + \C_{\err},
\]
where the first equality emphasizes that $\C_\SS$ depends on $\FF_\SS$ only through $N_\SS$, and $\overset{d}{=}$ denotes equality in distribution (given $N_\SS$).
\end{enumerate}
\end{prop}
\begin{proof}
Recall from Definition \ref{C_a} that for $k_\SS < a \le k$,
\[
C_a(t) = \sum_{i=1}^{N(t)} \mathbf{1}\{\sigma_i = a\} \, \eta_i \, (-1)^{\sum_{j>i} \mathbf{1}\{\sigma_j \in [k_\SS]\}}.
\]
Each jump of $(C_a(t))_{t \ge 0}$ occurs according to an independent Poisson process of rate $1/k$. At each jump, the sign 
\[
\eta_i \, (-1)^{\sum_{j>i} \mathbf{1}\{\sigma_j \in [k_\SS]\}}
\] 
is a $\Unif(\{\pm 1\})$ random variable, independent of the rest of the dynamics (including $\{\sigma_i\}_{i \ge 1}$), since the $\{\eta_i\}_{i \ge 1}$ are independently sampled uniform signs. Consequently, $\C_\RR$ is independent of $\C_\SS$ and of $\FF_\SS$. In particular, note that for $k_\SS<a\leq k$, $C_a$ evolves as an independent simple random walk on $\ZZ$ at rate $1/k$, and hence $\C_\RR=(C_a)_{k_\SS<a\leq k}$ is a rate $\rho_\RR=k_\RR/k$ simple random walk on $\ZZ^{k_\RR}$.

To show (ii), note that for $a\in [k_\SS]$,
\begin{align*}
C_a(t)&=(-1)^{(N_\SS\mod 2)}\sum_{i=1}^{N_\SS} (-1)^i\1\{ \tilde\sigma_i=a\}\\
&=(-1)^{(N_\SS\mod 2)}\left( \sum_{i\in [N_\SS]\cap 2\ZZ}\1\{ \tilde\sigma_i=a\}-\sum_{i\in [N_\SS]\cap (2\ZZ)^c}\1\{ \tilde\sigma_i=a\}\right)
\end{align*}
where the first term in the parentheses accounts for the even-numbered arrivals, while the second term accounts for the odd-numbered arrivals. At each arrival $i\in[N_\SS]$ an index $\tilde\sigma_i\in [k_\SS]$ is chosen independently from $\FF_\SS$ and uniformly at random from $[k_\SS]$, and hence
$$\C_\SS|\FF_\SS\overset{d}{=}\C^+_{\SS}-\C^-_{\SS}+\C_{\err}.$$
In particular, $\C_\SS$ depends on $\FF_\SS$ only through the value of $N_\SS$.
\end{proof}

\section{Entropy of Auxiliary Process}\label{sec:entropy_auxiliary_process}
The goal of this section is to obtain entropic estimates for the auxiliary process $\C=(\C_\SS,\C_\RR)$. We establish asymptotics of entropy for $1 \ll k \lesssim \log|G|$ and show concentration of entropy in the following two regimes, which are treated separately due to the need for different approaches:
\begin{enumerate}[label=(\roman*)]
    \item $k \gg 1$ with $k^2 \ll |G|^{2/k}$,
    \item $k \gg |G|^{2/k} (\log k)^2$ with $k \lesssim \log|G|$.
\end{enumerate}
This section contains technical entropic estimates and may be skipped on a first reading without interrupting the flow of the main proof.

\subsection{Entropy of $\C_\RR$}
%
Let $\mu^\RR_t$ denote the law of $\C_\RR(t)$, and define the random entropy
\[
Q_\RR(t) := -\log \mu^\RR_t(\C_\RR(t)).
\] 
As established in Proposition~\ref{auxiliary_law}, the process $\C_\RR$ is independent of  $\C_\SS$ and evolves as a simple random walk on $\mathbb{Z}^{k_\RR}$ with rate $\rho_\RR = k_\RR/k$. The entropic properties of high-dimensional simple random walks were analyzed in~\cite{hermon2018supplementary} (see Propositions~A.9 and~A.10), and are summarized below.

Since $\C_\RR$ is a rate-$\rho_\RR$ simple random walk on $\mathbb{Z}^{k_\RR}$, its entropy is obtained by summing the entropies of $k_\RR$ i.i.d.\ rate-$1/k$ simple random walks on $\mathbb{Z}$. Hence, the entropy of $\C_\RR$ is given by 
\[
\mathbb{E}[Q_\RR(t)] = k_\RR \cdot h(t/k),
\]
where $h:\mathbb{R}_+ \to \mathbb{R}_+$ denotes the entropy of a rate-$1$ simple random walk on $\mathbb{Z}$, whose asymptotics were studied in \cite{hermon2018supplementary} and are included below for completeness. It is well known that $h$ is strictly increasing and differentiable, see e.g. \cite[Proposition A.10]{hermon2018supplementary}. 
\begin{align*}
h'(s)&=-\sum_{x\in \ZZ} \frac{d}{ds} \P(W_s=x)\log \P(W_s=x)\\
&=-\sum_{x\in \ZZ}\left( \frac{1}{2}\P(W_s=x+1)+\frac{1}{2}\P(W_s=x-1)-\P(W_s=x)\right)\log \P(W_s=x)
\end{align*}
where $W_s$ denotes a rate-1 simple random walk on $\ZZ$. 
\begin{lemma}[Proposition  A.9 \cite{hermon2018supplementary}]\label{entropy_supplementary}
Let $h:\mathbb{R}_+ \to \mathbb{R}_+$ denote the entropy of a rate-$1$ (directed or undirected) random walk on $\ZZ$.  For $s\gtrsim 1$,
\beq\label{SRW_entropy_asymptotics}
h(s)=\frac{1}{2}\log(2\pi e s)+O(s^{-1/4}).
\eeq
\end{lemma}

Let $Q^{\mathrm{SRW}}(t)$ denote the entropy of a rate-1 simple random walk on $\ZZ$. 
The varentropy $\mathrm{Var}(Q^{\mathrm{SRW}}(\cdot))$ was analyzed in Proposition~A.6 of \cite{hermon2018supplementary}. 
Upon noting that 
\[
\mathrm{Var}(Q_\RR(t)) = k_\RR \cdot \mathrm{Var}(Q^{\mathrm{SRW}}(t/k))
\]
we can obtain directly from \cite[Proposition~A.6]{hermon2018supplementary} that there exists some constant $\beta>0$ such that $ \mathrm{Var}(Q^{\mathrm{SRW}}(s))\leq \beta$ for any $s\geq 0$,  and hence
\beq\label{varentropy_ub_CR}
\mathrm{Var}(Q_\RR(t)) \leq \beta k_\RR.
\eeq

We now turn to a concentration estimate for $Q_\RR$. Related results appear in \cite[Proposition~A.3]{hermon2018supplementary} and \cite[Proposition~2.3]{hermon2019cutoff}, with \cite[Proposition~A.3]{hermon2018supplementary} establishing a central limit theorem. Since our cutoff time $t_0(k,G)$ differs from that in~\cite{hermon2018supplementary} and our approach does not rely on a full central limit theorem, we instead establish an entropic concentration result  tailored  to our setting, giving a self-contained proof for completeness.

\begin{prop}\label{CR_concentration}
Assume that  $1\ll k \lesssim \log|G|$ and $1\ll \omega \ll k$.  For arbitrary $\ep>0$,
\begin{enumerate}[label=(\roman*)]
\item If $t\leq (1-\ep)t_0(k,G)$
$$  \P(Q_\RR(t)\leq \E[Q_\RR(t_0)]-\omega)=1-o(1).$$
\item  If $t\geq (1+\ep)t_0(k,G)$. Then 
  $$\P(Q_\RR(t)\geq  \E[Q_\RR(t_0)]+\omega)=1-o(1).$$
 \end{enumerate}
\end{prop}
\begin{proof}
For simplicity, write $t_0=t_0(k,G)$. It suffices to prove (i), as the proof of (ii) is analogous. 
For $t \leq (1-\ep)t_0$, we claim that there exists a  constant $c_\ep > 0$ such that
\beq\label{difference_entropy}
\E[Q_\RR(t_0)] - \E[Q_\RR(t)] \geq c_\ep k_\RR.
\eeq
Since $\E[Q_\RR(t)]$ is monotone increasing, it is enough to verify \eqref{difference_entropy} for $t = (1-\varepsilon)t_0$. 

Note that
\[
\E[Q_\RR(t)] = k_\RR \cdot h(t/k).
\]
In the regime $k \ll \log |G|$ (which implies $k \ll t_0$), the asymptotics in \eqref{SRW_entropy_asymptotics} directly yield \eqref{difference_entropy}. 
In the regime $k \asymp \log |G|$ (which implies $k \asymp t_0$), the mean value theorem gives
\[
\E[Q_\RR(t_0)] - \E[Q_\RR(t)] = k_\RR \cdot \big(h(t_0/k) - h(t/k)\big) = k_\RR \cdot h'(s) \frac{\varepsilon t_0}{k}
\]
for some $s \in (t/k, t_0/k)$. As $h$ is strictly increasing, we have $h'(s)>0$, which establishes \eqref{difference_entropy}.

For $t \leq (1-\ep)t_0$, combining \eqref{difference_entropy} and \eqref{varentropy_ub_CR} with $1 \ll \omega \ll k$, we obtain for sufficiently large $k_\RR$,
\begin{align*}
 \P\!\left(Q_\RR(t) > \E[Q_\RR(t_0)] - \omega\right) 
 &\leq \P\!\left(Q_\RR(t) - \E[Q_\RR(t)] > \E[Q_\RR(t_0)] - \E[Q_\RR(t)] - \omega\right) \\
 &\leq \P\!\left(Q_\RR(t) - \E[Q_\RR(t)] > c_\ep k_\RR/2\right) \\
 &\leq \frac{\Var(Q_\RR(t))}{(c_\ep k_\RR/2)^2} \;=\; o(1).
\end{align*}
\end{proof}

\subsection{Entropy Asymptotics for $\C_\SS$}
In what follows, we focus on the entropy of $\C_\SS$. Proposition~\ref{auxiliary_law} shows that the conditional law of $\C_\SS$ given $N_\SS$ can be described as the difference of two independent multinomial random variables. While the entropy of a single multinomial variable is well understood through combinatorial methods (see, e.g.,~\cite{kaji2016converging}), obtaining sharp entropic estimates for the difference is substantially more intricate and seems combinatorially intractable. We will use a new comparison argument to derive the asymptotics of the entropy of $\C_\SS$.

\subsubsection{Definitions and notations}
To keep the presentation self-contained, we begin by recalling some basic definitions and results relevant to our discussion; see, e.g., \cite{cover2006elements} for further reference. For simplicity of presentation, we restrict attention to discrete random variables.

Let $X$ be a random variable on a countable space $\mathcal{X}$ with law $P_X$.  
The entropy of $X$ is denoted by  
\[
H(X) := - \sum_{x \in \mathcal{X}} P_X(x) \log P_X(x)=\E[-\log P_X(X)].
\]
For a pair of random variables $(X,Y)$ taking values in $\mathcal{X}\times\mathcal{Y}$ with joint distribution  
$P_{X,Y}(x,y) = \mathbb{P}(X=x, Y=y)$, the \emph{joint entropy} of $X$ and $Y$ is  
\[
H(X,Y) := - \sum_{x \in \mathcal{X}} \sum_{y \in \mathcal{Y}} P_{X,Y}(x,y) \log P_{X,Y}(x,y).
\]

The \emph{conditional entropy} of $X$ given $Y$ is defined as 
\[
H(X|Y) := H(X,Y) - H(Y),
\]
and the \emph{mutual information} between $X$ and $Y$ is 
\[
I(X;Y) := D_{\mathrm{KL}}\!\big(P_{X,Y}\,\|\, P_X \otimes P_Y\big),
\]
where $D_{\mathrm{KL}}$ denotes the Kullback--Leibler divergence; equivalently, we have
\[
I(X;Y) = H(X) - H(X|Y) = H(Y) - H(Y|X).
\]
It is well known that mutual information is always nonnegative.  
In particular, this implies the useful inequality $H(X) \geq H(X|Y)$.

In addition to these properties, we will use the fact that if a random variable $X$ has finite support $\mathrm{supp}(X)$, then 
\beq\label{entropy_bound_support}
H(X) \le \log \left(|\mathrm{supp}(X)|\right),
\eeq
which follows from the fact that $D_{\mathrm{KL}}(P_X\| \Unif(\mathrm{supp}(X)))\geq 0$.

\medskip
For an event $A$, with slight abuse of notation we define the entropy of $X$ conditioned on $A$ as  
\beq
H(X | A) := - \sum_x \P(X = x | A)\, \log \P(X = x | A).
\eeq
It can be readily verified from definition that 
\beq\label{entropy_equality}
H(X,\1_A)=H(\1_A)+P(A)H(X|A)+P(A^c)H(X|A^c),
\eeq
where $H(\1_A)=-\P(A)\log \P(A)-\P(A^c)\log \P(A^c)$.

The joint entropy of a pair of random variables $(X,Y)$ conditioned on an event $A$ is defined to be 
\beq\label{conditional_joint_entropy}
H(X,Y|A)=-\sum_{x,y}\P(X=x,Y=y|A)\log\P(X=x,Y=y|A).
\eeq
It is straightforward to check that 
\beq\label{joint_entropy_equality}
H(X,Y|A)=\sum_{x} \P(X=x|A)H(Y|\{X=x\}\cap A)+H(X|A).\\
\eeq
This also implies $H(X|A)\leq H(X,Y|A)$ and $H(Y|A)\leq H(X,Y|A)$.

\subsubsection{Asymptotics of entropy}\label{sec: asymptotics of entropy}
Intuitively, the process $\C_\SS$ resembles a simple random walk on $\mathbb{Z}^{k_\SS}$, with the key difference that the number of positive and negative steps differs at most by 1 (see Lemma~\ref{auxiliary_law}). This global constraint introduces correlations among the coordinates and effectively reduces the intrinsic dimension of $\C_\SS$ to $k_\SS - 1$.

As will be seen later, this dimensional reduction accounts for why mixing occurs around 
\[
t_0(k,G) \;=\; \frac{k}{2\pi e}\,|G|^{2/(k-1)}
\]
in the regime $1 \ll k \ll \log |G|$, rather than at the entropic time
$
t_{\mathrm{ent}}(k,G) \;\eqsim\; \frac{k}{2\pi e}\,|G|^{2/k}.
$

We begin by analyzing the entropy of the difference of two i.i.d.\ multinomial random variables, as stated in Proposition~\ref{entropy_approximation}.

\begin{prop}\label{entropy_approximation}
Let $h:\mathbb{R}_+ \to \mathbb{R}_+$ denote the entropy of a rate-1 simple random walk on $\ZZ$.  
Suppose $d,N\in\mathbb{N}$ are such that $1 \ll d \lesssim N$.  Let $X$ and $Y$ be independent $\mathrm{Multi}_d(N, \vp)$ random variables with uniform parameter $\vp = \tfrac{1}{d}(1,\dots,1) \in \mathbb{R}^d$.
Then the entropy of their difference satisfies
$$ H(X-Y)= d\cdot h\left(\frac{2N}{d}\right)+
\begin{cases}
 \max\{ O(dN^{-1/3}), O(d(N/d)^{-1/4}),O(\log N)\} &\text{ if $1\ll d \ll N$,}\\
 O(N^{2/3}) &\text{ if $d \asymp N$.}
\end{cases}
$$
where the $O(\cdot)$ term indicates that the difference can be either positive or negative.
\end{prop}
\begin{remark}
In the regime $d \asymp N$, even for $\mathrm{Multi}_d(N,\vp)$, estimating the entropy via combinatorial methods is challenging; see, for example, Corollary 4.2 of \cite{kaji2016converging}, which requires $N \ge 2d$. 
Our result applies in the regime $d \asymp N$ and can be adapted to $\mathrm{Multi}_d(N,\vp)$, thereby extending the  result in \cite{kaji2016converging}.
\end{remark}

\begin{proof}
We begin with the proof of upper bound, which uses a comparison argument. Let 
$N_X, N_Y$ be independent $\mathrm{Poisson}(N(1 + N^{-1/3}))$ random variables, and define
\[
W := \mathrm{Multi}_{d}\!\left(N_X, \vp\right) 
   - \mathrm{Multi}_{d}\!\left(N_Y, \vp\right),
\]
where the two multinomial variables are independent. By a standard Poissonization argument, each coordinate of $W$ is the difference of independent Poisson variables, so $W$ has the law of a rate-$1$ simple random walk on $\ZZ^d$ run for time $2N(1+N^{-1/3})$. We will first show, via a comparison argument, that for sufficiently large $N$,
\begin{equation}\label{entropy_comparison_ub}
H(X-Y) \le \left(1 - 2 e^{-N^{1/3}/4}\right)^{-1}( H(W) + 1)\leq (1+4e^{-N^{1/3}/4})H(W)+2.
\end{equation}

Define the event $A := \{N_X > N,\, N_Y > N\}$. Conditioned on $(N_X, N_Y)$, further define
\[
R_X \sim \mathrm{Multi}_{d}\left((N_X - N)\1_A,\, \vp\right), 
\qquad
R_Y \sim \mathrm{Multi}_{d}\left((N_Y - N)\1_A,\, \vp\right),
\]
which are independent of each other given $(N_X, N_Y)$, and jointly independent of $X$ and $Y$.

Let 
\[
Z := X - Y + R_X - R_Y.
\]
Since $R_X-R_Y$ is independent of $X-Y$, we have $H(Z)\geq H(Z | R_X-R_Y)=H(X-Y).$
Moreover, we observe that
\begin{itemize}
\item The conditional law of $Z$ given $A$ coincides with that of $W$ given $A$. It follows that
$
H(Z | A) = H(W | A).
$
\item The conditional law of $Z$ given $A^c$ coincides with the (unconditional) law of $X-Y$, since $X-Y$ is independent of $A^c$. Hence,
$
H(Z | A^c) = H(X-Y).
$
\end{itemize}

Hence, applying the equality in \eqref{entropy_equality}, we obtain
\begin{align*}
H(X-Y) &\leq H(Z) \leq H(Z;\1_A) \\
&= H(\1_A) + \P(A)\,H(Z | A) + \P(A^c)\,H(Z | A^c) \\
&\leq H(\1_A) + \P(A)\,H(W | A) + \P(A^c)\,H(X-Y) \\
&\leq H(\1_A) +H(W) + \P(A^c)\,H(X-Y),
\end{align*}
where the  fourth line follows upon noting $H(\1_A)+\P(A)H(W| A)\leq H(W,\1_A)\leq H(\1_A) +H(W)$.
Rearranging gives
\beq\label{comparison_ub_1}
H(X-Y) \leq (1-\P(A^c))^{-1}\left(H(W) + H(\1_A)\right).
\eeq
It follows from a Chernoff bound that $\P(A^c) \leq 2\,\P(N_X \leq N) \leq 2 \exp(- N^{1/3}/4)$.
Thus, for sufficiently large $N$, the entropy of the indicator satisfies $H(\1_A) = -\P(A)\log \P(A) - \P(A^c)\log \P(A^c) \leq 1$. Combining these estimates with \eqref{comparison_ub_1} gives the upper bound stated in \eqref{entropy_comparison_ub}.

It remains to understand $H(W)$.
Since $W$ has the law of a rate-$1$ simple random walk on $\ZZ^{d}$ run for time $2N(1+N^{-1/3})$, with each coordinate evolving independently, we immediately have
\[
H(W) =d\cdot h\left(\frac{2N(1+N^{-1/3})}{d}\right).
\]
For $1 \ll d \ll N$, substituting the asymptotics from \eqref{SRW_entropy_asymptotics} gives
$$
h\left(\frac{2N(1+N^{-1/3})}{d}\right)
= h\left(\frac{2N}{d}\right) + \tfrac12 \log(1+N^{-1/3}) + O((N/d)^{-1/4}).
$$
If $d \eqsim \lambda N$ for some $\lambda>0$, the mean value theorem implies
$$
h\left(\frac{2N(1+N^{-1/3})}{d}\right)
= h\left(\frac{2N}{d}\right) +\frac{2N^{2/3}}{d}h'(s)
$$
for some $s\in \left( \frac{2N}{d}, \frac{2N(1+N^{-1/3})}{d}\right)$ which satisfies $h'(s)\in (0,\infty)$.

This upper bounds $H(W)$, i.e.,
\beq\label{lb_entropy_W}
H(W) \leq d\cdot h\left(\frac{2N}{d}\right)+
\begin{cases}
d\cdot \max\{O(N^{-1/3}), O((N/d)^{-1/4})\} & \text{ if $1\ll d \ll N$,}\\
O(N^{2/3})& \text{ if $d\asymp N$,}\\
\end{cases}
\eeq
where the $O(\cdot)$ terms are positive. This, together with \eqref{entropy_comparison_ub}, establishes the desired upper bound on $H(X-Y)$.

\bigskip

The lower bound follows by a similar approach, with differences only in certain estimates.
  Let $\widetilde N_X, \widetilde N_Y$ be independent $ \mathrm{Poisson}(N(1 - N^{-1/3}))$ random variables and define
\[
\widetilde W := \mathrm{Multi}_{d}\!\left(\widetilde N_X, \vp\right) 
   - \mathrm{Multi}_{d}\!\left(\widetilde N_Y,\vp\right),
\]
where the two multinomial variables are independent. Define the event $B = \{\widetilde N_X <N,\, \widetilde N_Y < N\}$, and let  
\[
\widetilde R_X \sim \mathrm{Multi}_{d}\left( (N-\widetilde N_X)\1_B,\, \vp\right), 
\qquad 
\widetilde R_Y \sim \mathrm{Multi}_{d}\left( (N-\widetilde N_Y)\1_B,\, \vp\right),
\]
which are independent of each other given $(N_X, N_Y)$, and jointly independent of $X$ and $Y$.

Let
\[
\widetilde Z := \widetilde W + \widetilde R_X - \widetilde R_Y.
\]
Observe that the conditional law of $\widetilde Z$ given $B$ coincides with the law of $X-Y$, while the conditional law of $\widetilde Z$ given $B^c$ coincides with the conditional law of $\widetilde W$ given $B^c$. This implies
\beq\label{conditional_laws_lb}
H(\widetilde Z | B) = H(X-Y)\quad \text{ and }\quad H(\widetilde Z | B^c) = H(\widetilde W | B^c).
\eeq
Since $\widetilde W$ and $\widetilde R_X-\widetilde R_Y$ are conditionally independent given $(\widetilde N_X,\widetilde N_Y)$, we have
\begin{align*}
H(\widetilde Z|(\widetilde N_X,\widetilde N_Y) )&=H(\widetilde W+\widetilde R_X-\widetilde R_Y|(\widetilde N_X,\widetilde N_Y))\geq H(\widetilde W| (\widetilde N_X,\widetilde N_Y))\geq H(\widetilde W)-H(\widetilde N_X,\widetilde N_Y).
\end{align*}
On the other hand, by \eqref{entropy_equality} and \eqref{conditional_laws_lb}, we have
\begin{align*}
H(\widetilde Z|(\widetilde N_X,\widetilde N_Y) )&\leq H(\widetilde Z)\leq H(\widetilde Z,\1_B)=H(\1_B)+\P(B)H(X-Y)+\P(B^c)H(\widetilde W|B^c)
\end{align*}
Combining the above inequalities and rearranging yields
\beq\label{entropy_comparison_lb}
H(X-Y)\geq \P(B)^{-1}\left( H(\widetilde W)-H(\widetilde N_X,\widetilde N_Y)-H(\1_B)-\P(B^c)H(\widetilde W|B^c)\right).
\eeq
Since $\widetilde W$ has the law of a rate-$1$ simple random walk on $\ZZ^{d}$ at time $2N(1-N^{-1/3})$, an argument analogous to \eqref{lb_entropy_W} gives the lower bound
$$
H(\widetilde W)\geq d\cdot h\left(\frac{2N}{d}\right)-
\begin{cases}
d\cdot \max\{O(N^{-1/3}), O((N/d)^{-1/4})\} & \text{ if $1\ll d \ll N$,}\\
O(N^{2/3})& \text{ if $d\asymp N$,}\\
\end{cases}
$$
where the $O(\cdot)$ terms are positive. 

We will control the remaining terms in \eqref{entropy_comparison_lb}. A Chernoff bound gives $\P(B^c) \leq 2 \exp(-N^{1/3}/4)$, which implies $H(\1_B) \leq 1$ for sufficiently large $N$.  
Treating a Poisson random variable as the position of a directed random walk, Lemma~\ref{entropy_supplementary} yields
\beq\label{poisson_entropy}
H(\widetilde N_X, \widetilde N_Y) = H(\widetilde N_X) + H(\widetilde N_Y) = 2\, h\bigl(N(1-N^{-1/3})\bigr) \leq 2 \log(2\pi e N)
\eeq
for large $N$. Thus, the desired lower bound follows once we establish the following estimate
\beq\label{residual_entropy}
\P(B^c)\, H(\widetilde W | B^c)\le 4\log N
\eeq
for sufficiently large $N$. We now verify this bound to conclude the proof.

Let $\widetilde N=(\widetilde N_X,\widetilde N_Y)$. Recall the definition of conditional joint entropy from \eqref{conditional_joint_entropy}, and it follows from \eqref{joint_entropy_equality} that
We can write
\begin{align}\label{entropy_error_bound}
\nonumber H(\widetilde W|B^c)&\leq H(\widetilde W,\widetilde N |B^c)\\
&=\sum_{(n_1,n_2): \max\{n_1,n_2\}\geq N} \P(\widetilde N=(n_1,n_2)|B^c)H(\widetilde W|\{\widetilde N=(n_1,n_2)\}\cap B^c)+H(\widetilde N|B^c)
\end{align}
Conditioned on $\widetilde N=(n_1,n_2)$, the support of $\widetilde W$ has size at most 
${n_1+d-1 \choose d-1}{n_2+d-1\choose d-1}\leq (n_1+d)^{n_1+d}(n_2+d)^{n_2+d}$,
and by \eqref{entropy_bound_support},
$$H(\widetilde W|\{\widetilde N=(n_1,n_2)\}\cap B^c)\leq \log \left((n_1+d)^{n_1+d}(n_2+d)^{n_2+d}\right).$$
It then follows that 
\begin{align*}
 &\sum_{(n_1,n_2): \max\{n_1,n_2\}\geq N} \P(\widetilde N=(n_1,n_2)|B^c)H(\widetilde W|\{\widetilde N=(n_1,n_2)\}\cap B^c)\\
\leq & \E[(\widetilde N_X+d)\log(\widetilde N_X+d)+(\widetilde N_Y+d)\log(\widetilde N_Y+d)|B^c]\\
=&2\E[(\widetilde N_X+d)\log(\widetilde N_X+d)|B^c].
\end{align*}
Applying this estimate in \eqref{entropy_error_bound} yields, when $N$ is sufficiently large,
\begin{align*}
\P(B^c)H(\widetilde W|B^c)&\leq 2\E[\1_{B^c}(\widetilde N_X+d)\log(\widetilde N_X+d) ]+\P(B^c)H(\widetilde N|B^c)\\
&\leq 2\E[ \1\{\widetilde N_X\geq N\} (\widetilde N_X+d)\log(\widetilde N_X+d)]\\
&\quad  +2\E[ \1\{\widetilde N_X< N,\widetilde N_Y\geq N\}(\widetilde N_X+d)\log(\widetilde N_X+d)]+H(\widetilde N, \1_B)\\
&\leq 2\E[ ((\widetilde N_X+d)\log(\widetilde N_X+d))^2]^{1/2}\P(\widetilde N_X\geq N)^{1/2}\\
&\quad+2(N+d)\log(N+d)\P(\widetilde N_Y\geq N)+H(\widetilde N)+H(\1_B)\\
&\leq 4(N+d)^2e^{-N^{1/3}/8}+2(N+d)^2e^{-N^{1/3}/4}+2 \log(2\pi e N)+1\\
&\leq 4\log N.
\end{align*}
The third inequality follows from the Cauchy--Schwarz inequality. The fourth inequality uses the fact that $\log x \le x$ for all $x>0$ and the deterministic bound
\[
\E\big[( (\widetilde N_X+d)\log(\widetilde N_X+d) )^2\big] \le \E\big[(\widetilde N_X + d)^4\big] \le 4 (N+d)^4
\]
for sufficiently large $N$ and $d$. Moreover, we have $H(\widetilde N) \le 2 \log(2\pi e N)$ from \eqref{poisson_entropy}, $\P(B^c) \le 2 e^{-N^{1/3}/4}$, and
\[
H(\1_B) = -\P(B)\log \P(B) - \P(B^c)\log \P(B^c) \le 1.
\]
Combining these estimates completes the proof of \eqref{residual_entropy}, and hence establishes the desired lower bound.
\end{proof}
 
 We now apply Proposition~\ref{entropy_approximation} to derive sharp entropy asymptotics for $\C_\SS$ conditioned on $N_\SS$. 
To describe the conditional law of $\C_\SS$ given $N_\SS$, we introduce the discrete-time chain $(Y_m)_{m \ge 0}$ on $\ZZ^{k_\SS}$ as follows.

Let $\{\delta_j\}_{j \ge 1}$ be i.i.d.\ uniform random elements of $\{e_i : i \in [k_\SS]\}$. Define the chain $Y_m$ starting at $Y_0 = 0$, with updates for $m \ge 1$ given by
\begin{equation}\label{CDG_process}
Y_m = -Y_{m-1} + \delta_m.
\end{equation}
It is straightforward to see from Lemma \ref{auxiliary_law} that, conditioned on $N_\SS$, there exists an almost sure coupling
\begin{equation}\label{coupling_CDG_process}
\C_\SS = Y_{N_\SS}.
\end{equation}

\begin{corollary}\label{Y_entropy}
Let $H(Y_m)$ denote the entropy of $Y_m$. Suppose $1 \ll k \lesssim m$. Then
\[
H(Y_m) =k_\SS \cdot h\left(\frac{m}{k_\SS}\right)+
\begin{cases}
 \max\{ O(km^{-1/3}), O(k(m/k)^{-1/4}),O(\log m)\} &\text{ if $1\ll k \ll m$,}\\
 O(m^{2/3}) &\text{ if $k \asymp m$,}
\end{cases}
\]
where the $O(\cdot)$ term indicates that the difference can be either positive or negative.
\end{corollary}

\begin{proof}
Grouping the positive and negative steps in $Y_m$ separately, one can easily see that
\[
Y_m = Y^+ - Y^- + Y_{\mathrm{err}},
\]
where $Y^\pm \sim \mathrm{Multi}_{k_\SS}(\lfloor m/2 \rfloor, \vp)$ and $Y_{\mathrm{err}} \sim \mathrm{Multi}_{k_\SS}(m \bmod 2, \vp)$ are independent multinomial random variables, with 
\(\vp = \frac{1}{k_\SS}(1,1,\dots,1) \in \mathbb{R}^{k_\SS}\).
 Hence,
\begin{align*}
H(Y^+-Y^-)+H(Y_{\err})&=H(Y^+-Y^-,Y_{\err})\geq H(Y^+-Y^-+Y_{\err})\\
&=H(Y_m)\geq H(Y_m|Y_{\err})=H(Y^+-Y^-),
\end{align*}
which implies
$$|H(Y_m)-H(Y^+-Y^-)|\leq H(Y_{\err})\leq \log k_\SS.$$
The conclusion then follows by applying Lemma~\ref{entropy_approximation} to $H(Y^+-Y^-)$
with $d = k_\SS $ and $N = \lfloor m / 2 \rfloor$.
\end{proof}

\subsection{Regime $k\gg 1$ with $k^2\ll |G|^{2/k}$}\label{sec:entropy_small_regime}
The main goal of this section is to establish the concentration of the entropy of $\C_\SS$ in the regime $k \gg 1$ with $k^2 \ll |G|^{2/k}$, as stated in Proposition~\ref{CS_entropy_concentration}. 
The proof relies on several preliminary estimates.  

We begin by presenting useful results for multivariate normal distributions in Section~\ref{multivariate_normal}, 
and then apply these results to analyze the entropy of the auxiliary process $\C_\SS$ in Section~\ref{sec:entropy_auxiliary}.
The exposition in this section is technical, since we need to carefully approximate $\C_\SS$ by multivariate normal random variables to establish its concentration within typical entropic sets. 
Along the way, we will aim to explain the purpose of each lemma and proposition to provide an intuitive guide through the technical arguments.

\bigskip

Throughout this section we will set  $d=k_\SS-1$ and let 
$$\Sigma=\Sigma_d:=\mathrm{diag}(\hat \vp)-\hat \vp \hat \vp^T \quad\quad \text{ where }  \hat \vp=\frac{1}{d+1}\cdot(1,\dots,1)  \in \mathbb{R}^d.$$ 
Let $\phi_\Sigma$ denote the density function of the $d$-dimensional $\mathrm{Normal}(0,\Sigma)$ random variable, i.e.,
\beq\label{pdf_normal}
\phi_{\Sigma}(x)=(2\pi)^{-d/2}|\Sigma|^{-1/2} \exp\left(-\frac{x^T \Sigma^{-1} x}{2}\right) \quad \text{ for }x\in \mathbb{R}^d.
\eeq
It is well known (see e.g. Theorem 1 in \cite{tanabe1992exact}) that $|\Sigma|=(d+1)^{-(d+1)}$ and 
\beq\label{Sigma_inverse}
\Sigma^{-1}=(d+1)(I+\mathbf{1}\cdot \mathbf{1}^T).
\eeq
We consider multivariate normal distributions because, in this regime, the multinomial variables $\C_\SS^{\pm}$ are well approximated by them, allowing us to study the entropic concentration of $\C_\SS$ via this approximation.

\subsubsection{Concentration of entropy for multivariate normal}\label{multivariate_normal}

For  $1\ll d\ll m^{1/2}$, let $\xi_m\sim \mathrm{Normal}_d(0,m\Sigma)$ be a $d$-dimensional multivariate normal random variable and $\phi_{\Sigma,m}$ denote its probability density function. Note from \eqref{pdf_normal} that 
$$\phi_{\Sigma,m}(x)=m^{-d/2}\phi_{\Sigma}(x/\sqrt{m}).$$
For $m\in\mathbb{R}_+$ and $\delta>0$, define
\beq\label{normal_typical_set}
\WW^{\mathrm{Normal}}_{m,\delta}:=\left\{ x\in \mathbb{R}^d: \left(\frac{2\pi e m}{(1-\delta)(d+1)}\right)^{-d/2}\leq \phi_{\Sigma,m}(x)\leq \left(\frac{2\pi e m}{(1+\delta)(d+1)}\right)^{-d/2}\right\}.
\eeq

This set plays a key role in the proofs that follow. Intuitively, it captures where $\xi_m \sim \mathrm{Normal}_d(0, m \Sigma)$ typically lies (Lemma~\ref{typical_normal}). In the regime $k \gg 1$ with $k^2 \ll |G|^{2/k}$, the auxiliary process $\C_\SS$ is well approximated by a multivariate normal, so $\WW^{\mathrm{Normal}}_{m,\delta}$ with a suitable $m$ naturally describes a typical set for $\C_\SS$, as formalized in Proposition~\ref{CS_entropy_concentration}.

We begin by showing that $\WW^{\mathrm{Normal}}_{m,\delta}$ is a typical set for $\xi_m \sim \mathrm{Normal}_d(0, m \Sigma)$.

\begin{lemma}\label{typical_normal}
Suppose $1\ll d\ll m$. Let $\xi_m\sim \mathrm{Normal}_d(0,m\Sigma)$. For any $\delta>0$, we have
$$\P(\xi_m\in \WW^{\mathrm{Normal}}_{m,\delta})=1-o(1).$$
\end{lemma}
\begin{proof}
 Define the random entropy $Q(\xi_m):=-\log \phi_{\Sigma,m}(\xi_m)$. Note that $|m\Sigma|=m^d (d+1)^{-(d+1)}$. One can directly compute the entropy of $\xi_m$ to obtain
\begin{align*}
\E[Q(\xi_m)]&=\frac{d}{2}\log(2\pi e)+\frac{1}{2}\log|m\Sigma|=\frac{d}{2}\log(2\pi e)+\frac{d}{2}\log m-\frac{d+1}{2}\log (d+1) \\
&=\frac{d}{2}\log\left(\frac{2\pi  e m}{d+1}\right)-\frac{1}{2}\log(d+1).
\end{align*}
 As the multivariate normal density $\phi_{\Sigma,m}(\cdot)$ is log-concave on $\mathbb{R}^d$,  Theorem 1.1 in \cite{BobkovMadiman2011}  applies directly to yield a concentration bound for the random entropy.
 $$\P( |Q(\xi_m)-\E[Q(\xi_m)]|\geq r\sqrt{d})\leq 2e^{-r/16}.$$
It follows that
\begin{align*}
\P(\xi_m\notin \WW^{\mathrm{Normal}}_{m,\delta})&=\P\left(\phi_{\Sigma,m}(\xi_m)> \left(\frac{2\pi e m}{(1+\delta)(d+1)}\right)^{-d/2}\right)+\P\left(\phi_{\Sigma,m}(\xi_m)< \left(\frac{2\pi e m}{(1-\delta)(d+1)}\right)^{-d/2}\right)\\
&=\P(Q(\xi_m)<\E[Q(\xi_m)]+\frac{1}{2}\log (d+1)-\frac{d}{2}\log(1+\delta))\\
&\quad +\P(Q(\xi_m)>\E[Q(\xi_m)]+\frac{1}{2}\log (d+1)-\frac{d}{2}\log(1-\delta))\\
&\leq  \P( |Q(\xi_m)-\E[Q(\xi_m)]|> \frac{\delta d}{8})\leq 2e^{-\delta \sqrt{d}/128}=o(1).
\end{align*}

\end{proof}

As an immediate consequence of the entropy concentration, we obtain in Corollary \ref{normal_outlier} that, with high probability, $\xi_m$ lies in the bulk 
\beq\label{def_bulk_set}
\AA_{m,d}=\{ x\in \mathbb{R}^d:  \|x\|_2\leq d^{-1/2}m^{2/3}\}
\eeq
satisfying an $L^2$-norm constraint. The purpose of considering $\AA_{m,d}$ is to gain additional control over the variation of the density function $\phi_{\Sigma,m}(x)$
over small neighborhoods of $x$ and for small changes in $m$  (see Lemma~\ref{density_comparison}). Note that the technical assumption $1 \ll d^3 \ll m$ ensures that $\AA_{m,d}$ captures a typical region for $\xi_m$.

\begin{corollary}\label{normal_outlier}
Suppose $1 \ll d^3 \ll m$. 
For $\delta,\beta>0$, let $\WW^{\mathrm{Normal}}_{m,\delta}$ and $\AA_{\beta m,d}$ be defined as in~\eqref{normal_typical_set} and~\eqref{def_bulk_set}, respectively.  Then
\[
   \WW^{\mathrm{Normal}}_{m,\delta} \subseteq \AA_{\beta m,d}
   \qquad \text{for any fixed $\delta,\beta>0$.}
\]
Consequently, for $\xi_m \sim \mathrm{Normal}_d(0, m\Sigma)$ and any $\beta>0$,
$
   \P\bigl(\xi_m \in \AA_{\beta m,d}\bigr) \;=\; 1-o(1).
$

\end{corollary}

\begin{proof}
Recall from \eqref{Sigma_inverse} that  $\Sigma^{-1}=(d+1)(I+\mathbf{1}\cdot\mathbf{1}^T)$. Thus, any $x\in \AA^c_{\beta m,d}$ satisfies 
$$x^T \Sigma^{-1}x= (d+1)\left(\|x\|_2^2+  (\sum_{i=1}^d x_i)^2\right)\geq (d+1)\|x\|_2^2\geq  ( \beta m)^{4/3},$$
which implies that 
\begin{align*}
\phi_{\Sigma,m}(x)&=(2\pi m )^{-d/2}|\Sigma|^{-1/2} \exp\left( -\frac{ x^T \Sigma^{-1}x}{2m}\right)\\
&\leq  (d+1)^{1/2}e^{d/2}(1-\delta)^{-d/2}\exp(- \beta^{4/3} m^{1/3}/2)\cdot \left(\frac{2\pi e m}{(1-\delta)(d+1)}\right)^{-d/2}\\
&\leq   \left(\frac{2\pi e m}{(1-\delta)(d+1)}\right)^{-d/2}
\end{align*}
for any $\delta>0$ when $d$ is sufficiently large since $d^3\ll m$. That is, for any fixed $\delta>0$ and $\beta>0$, $\AA^c_{\beta m,d} \subseteq (\WW^{\mathrm{Normal}}_{m,\delta})^c$ when $d$ is sufficiently large.  The conclusion follows from Lemma \ref{typical_normal}.
\end{proof}

The lemma below provides a quantitative bound on how the normal density $\phi_{\Sigma,m}(x)$ varies with respect to $x$ and  $m$.

\begin{lemma}[Local stability of Normal density] \label{density_comparison}
Suppose $1\ll d^3\ll m$. For any $\eta>0$, any $x\in\AA_{m,d}=\{ x\in \mathbb{R}^d: \|x\|_2\leq d^{-1/2}m^{2/3}\}$  and  $z,z'\in [-2,2]^d$ we have
$$ \frac{\phi_{\Sigma,m}(x+z)}{\phi_{\Sigma,m}(x+z')}\leq e^{\eta d}$$
when $d$ and $m$ are sufficiently large. 

Moreover, if $m'\in [m-O(m^{7/12}), m+O(m^{7/12})]$ and $x\in \AA_{m,d}$, then for sufficiently large $d$ and $m$,
$$e^{-\eta d}\leq \frac{\phi_{\Sigma, m'}(x)}{\phi_{\Sigma, m}(x)}\leq e^{\eta d}.$$
\end{lemma}
\begin{proof}
 For the desired result it suffices to show that for any $x\in \AA_{m,d}$ and $z\in [-2,2]^d$, 
\beq\label{normal_density_comparison}
e^{-\eta d/2}\leq \frac{\phi_{\Sigma,m}(x)}{\phi_{\Sigma,m}(x+z)}\leq e^{\eta d/2}.
\eeq
Since $\Sigma$ is positive-definite, for any $z\in \mathbb{R}^d$ we have $z^T\Sigma^{-1} z\geq 0$. Then for $x\in \mathbb{R}^d$ and $z\in [-2,2]^d$,
\begin{align*}
\frac{\phi_{\Sigma,m}(x)}{\phi_{\Sigma,m}(x+z)}&=\exp\left( -\frac{1}{2} (x^T(m\Sigma)^{-1}x-(x+z)^T(m\Sigma)^{-1}(x+z))\right)\\
&=\exp\left(\frac{1}{2m}\left( 2z^T\Sigma^{-1} x-z^T\Sigma^{-1} z\right)\right)\leq \exp\left(\frac{z^T\Sigma^{-1} x}{m}\right).\end{align*}
Recall from \eqref{Sigma_inverse} that $\Sigma^{-1}=(d+1)(I+\1\cdot \1^T)$. Since $z\in [-2,2]^d$,
\begin{align*}
|z^T \Sigma^{-1} x|&=(d+1)\bigg| z^Tx+ (\sum_{i=1}^d z_i)(\sum_{i=1}^d x_i) \bigg|\leq (d+1)(2\|x\|_1+ 2d\|x\|_1)\leq 8d^2\|x\|_1.
\end{align*}
 It is easy to see that any $x\in \AA_{m,d}$ satisfies $\|x\|_1\leq d^{1/2}\|x\|_2\leq m^{2/3}$. 
Since $1\ll d^3\ll m$, when $x\in \AA_{m,d}$, 
$$\frac{\phi_{\Sigma,m}(x)}{\phi_{\Sigma,m}(x+z)}\leq \exp\left(\frac{|z^T\Sigma^{-1} x|}{m}\right)\leq \exp( 8d^2\|x\|_1/m)\leq \exp(8d^2m^{-1/3})=e^{o(d)}.$$

To prove the other side of the inequality, notice that when $d$ is sufficiently large we have 
$|z^T\Sigma^{-1}z|=(d+1)(z^Tz+(\sum_{i=1}^d z_i)^2)\leq  16d^3$ for $z\in [-2,2]^d$.
It follows in the same way as before that
\begin{align*}
\frac{\phi_{\Sigma,m}(x+z)}{\phi_{\Sigma,m}(x)}&=\exp\left(\frac{-2z^T\Sigma^{-1} x+z^T \Sigma^{-1}z}{2m}\right)\leq \exp\left(\frac{|z^T\Sigma^{-1} x|}{m}\right)\exp\left(\frac{|z^T\Sigma^{-1}z|}{2m}\right)\\
&\leq e^{o(d)+8d^3/m}\leq e^{\eta d/2}.
\end{align*}
This concludes the proof of \eqref{normal_density_comparison}.

The second inequality follows by a similar reasoning. As
$$|x^T\Sigma^{-1}x|= (d+1)\left(\|x\|_2^2+  (\sum_{i=1}^d x_i)^2\right)\leq (d+1)^2\|x\|_2^2\leq 2dm^{4/3},$$
for $m'=(1+\theta)m$ with $|\theta|=O( m^{-5/12})$ we have
\begin{align*}
\frac{\phi_{\Sigma, m'}(x)}{\phi_{\Sigma, m}(x)}&=\left(\frac{m'}{m}\right)^{-d/2}  \exp\left( -\frac{ x^T \Sigma^{-1}x}{2m'}+\frac{ x^T \Sigma^{-1}x}{2m}\right)\leq (1-|\theta|)^{-d/2}\exp\left( \frac{\theta |x^T\Sigma^{-1}x|}{2(1+\theta )m}\right)\\
&\leq \exp\left( O(dm^{-5/12})+O(dm^{-1/12}) \right)\leq e^{\eta d}
\end{align*}
for any $\eta>0$ when $d$ and $m$ are sufficiently large. 
Applying the same estimate to $\frac{\phi_{\Sigma, m}(x)}{\phi_{\Sigma, m'}(x)}$ gives the  lower bound.

\end{proof}

\subsubsection{Entropy of auxiliary process}\label{sec:entropy_auxiliary}

Recall that $d = k_\SS - 1$, and note that $\C_\SS \in \ZZ^{d+1}$ satisfies the constraint $\sum_{i=1}^{d+1} \C_{\SS,i} =\1\{ N_\SS \text{ is odd}\}$, as described in Proposition~\ref{auxiliary_law}. Recall from  Proposition~\ref{auxiliary_law} that 
$$\C_\SS=\C_\SS^+-\C_\SS^-+\C_\err.$$
Throughout this section, we use the superscript $\hat{\cdot}$ indicates projection to dimension $d$. For instance, $\hat{\C}_\SS \in \ZZ^d$ denotes the projection of $\C_\SS$ onto its first $d$ coordinates. Given $N_\SS$, this projection uniquely determines the original vector $\C_\SS$. Similarly, let $\hat \vp=(1,\dots,1)/(d+1)\in\mathbb{R}^d$.

We begin the normal approximation of $\C_\SS$ by introducing the following random variables. For a vector $x\in \mathbb{R}^d$, let $[x]$ denote its nearest integer vector in $\ZZ^d$. For $m \in \mathbb{N}$, let $\xi^+_m, \xi^-_m \overset{\mathrm{iid}}{\sim} \mathrm{Normal}_d\left(m\hat \vp,m\Sigma\right)$, and define
\beq\label{M_xi}
\hat{M}^\xi_{m} := [\xi^+_m] - [\xi^-_m] \in \ZZ^d,
\eeq
which corresponds uniquely to a vector $M^\xi_m\in \ZZ^{d+1}$ defined by
$$
M^\xi_m(i):=
\begin{cases}
\hat M^\xi_m(i) &\text{ for $i\in [d]$},\\
- \sum_{i=1}^d ( [\xi^+_m](i)-[\xi^-_m](i))&\text{ for $i=d+1$},
\end{cases}
$$
where the $(d+1)$-th coordinate comes from requiring $\sum_{i=1}^{d+1} M^\xi_m(i)=0$. The following coupling result is a key ingredient in later analysis, which will explain how $\C_\SS=\C_\SS(t)$ can be approximated by a multivariate normal random variable.

\begin{lemma}\label{coupling}
Suppose $1\ll d^2\ll t$. Let $M_{\SS}(t):=\sum_{m=0}^\infty \1_{\{\floor{N_\SS(t)/2}=m\}}M^\xi_m$. We can couple $M_\SS$ with $\C^+_\SS-\C^-_\SS$ with high probability, i.e.,
$$\P( M_\SS= \C^+_\SS-\C^-_\SS)=1-o(1).$$
\end{lemma}
\begin{proof}
Let $m\in \mathbb{N}$ and define random variables  $M^{\pm}_m\overset{\mathrm{iid}}{\sim} \mathrm{Multi}_{d+1}(m, \vp)$. Note that conditioned on $\{\floor{N_\SS/2}=m\}$, $\C^{\pm}_\SS$ have the same law as $M^{\pm}_m$.

To approximate with normal distribution, we consider the projection of $M^{\pm}_m$ to the first $d$ coordinates, denoted by $\hat M^{\pm}_m$. Letting $\hat \P_{m,\vp}$ denote the law of $\hat M_m^{\pm}-m\hat \vp$ and $\mathbb{Q}_{m,\vp}$ denote the law of $\mathrm{Normal}_d(0,m\Sigma)$, 
Lemma 3.1 in \cite{ouimet2021precise} states
$$\| \hat \P_{m,\vp}*\1_{[-1/2,1/2]^d}-\mathbb{Q}_{m,\vp}\|_{\mathrm{TV}}=O(m^{-1/2}d).$$
Hence, we can couple $\hat M_m^{\pm}-m\hat \vp$ with $[\xi^{\pm}_m]$ except with error probability $O(m^{-1/2}d)$, i.e., $\hat M_m^+-\hat M^-_m=[\xi^{+}_m]-[\xi^{-}_m]$ holds except with  probability $O(m^{-1/2}d)$. Hence, $\hat M_\SS$ can be constructed to satisfy
$$\P( \hat M_\SS=\hat \C^+_\SS-\hat \C^-_\SS|\floor{ N_\SS/2}=m)=1-O(m^{-1/2}d).$$

As $N_\SS$ follows $\mathrm{Poisson}(\rho_\SS t)$,  for any $\ep>0$, with high probability $N_\SS\in  [(1-\ep)\rho_\SS t,(1+\ep)\rho_\SS t]$. Summing over  typical values of $N_\SS$ then gives
\begin{align*}
\P( \hat M_\SS= \hat \C^+_\SS-\hat \C^-_\SS)&\geq \sum_{\ell \in [(1-\ep)\rho_\SS t,(1+\ep)\rho_\SS t]} \P(\hat M_\SS=\hat \C^+_\SS-\hat \C^-_\SS | N_\SS=\ell)\P(N_\SS=\ell)\\
&\geq (1-O(t^{-1/2}d))\cdot \P( N_\SS\in   [(1-\ep)\rho_\SS t,(1+\ep)\rho_\SS t])=1-o(1).
\end{align*}
Lastly, observe that due to the definition of $M_\SS$ and $\C^{\pm}_\SS$, the value of the last coordinate is uniquely determined by the  first $d$ coordinates, i.e.,
$$\{ M_\SS= \C^+_\SS-\C^-_\SS\}=\{\hat M_\SS= \hat \C^+_\SS-\hat \C^-_\SS\}$$
and the proof is complete.
\end{proof}

\bigskip
We are now ready to state the entropic concentration of $\C_\SS$ in the regime $k \gg 1$ with $k^2 \ll |G|^{2/k}$, formalized in Proposition \ref{CS_entropy_concentration}.
Recall from \eqref{cutoff_time} that in this regime
\[
t_0(k,G) = \frac{k}{2\pi e} |G|^{2/(k-1)},
\] 
and note that $1 \ll k^3 \ll t_0(k,G)$.

Define
\beq\label{typ_smallregime}
\typ^{\SS}:= \{ M_\SS= \C^+_\SS-\C^-_\SS\}\cap \{N_\SS\in  [\rho_\SS t-t^{7/12}, \rho_\SS t+t^{7/12}]\}.
\eeq
It is a direct consequence of Lemma \ref{coupling} and standard large deviations that 
\beq\label{typical_event_S_prob}
\P(\typ^\SS)=1-o(1).
\eeq
Recall that $d=k_\SS-1$. Let $\hat w$ denote the first $d$ coordinates of $w\in \ZZ^{d+1}$. We now define, for $\delta>0$ and $t\geq 0$,
\begin{align}\label{WS_smallregime}
\WW_\SS&=\WW_\SS(\delta,t):=\{ w\in \ZZ^{d+1}: \sum_{i=1}^{d+1}w_i\in \{0,1\}, \hat w\in \WW^{\mathrm{Normal}}_{\rho_\SS t,\delta}\cap \ZZ^d\}.
\end{align}

Our goal for the remainder of this section is to establish properties of $\WW_\SS(\delta,t)$ that will be useful in determining mixing times (see Lemma~\ref{WS_size} and Lemma~\ref{set_comparison}), and to show that $\C_\SS(t)$ typically belongs to $\WW_\SS(\delta,t)$ (see Proposition~\ref{CS_entropy_concentration}).

The following lemma provides an upper bound on the size of $\WW_\SS(\delta,t)$ before the proposed mixing time. Intuitively, this indicates that $\C_\SS$ remains confined to a relatively small set, and hence the walk cannot yet be well mixed.
\begin{lemma}\label{WS_size}
Suppose $1 \ll k^3 \ll t_0(k,G)$ set $d=k_\SS-1$. For any fixed $\ep > 0$ and $\delta>0$ sufficiently small relative to $\ep$, we have, for $t=(1-\ep)t_0(k,G)$,
\[
\bigl|\WW_\SS(\delta,t)\bigr|
= 2\bigl|\WW^{\mathrm{Normal}}_{\rho_\SS t,\delta} \cap \ZZ^d\bigr|
\;\leq\; |G|^{d/(k-1)}.
\]
\end{lemma}

\begin{proof}
For ease of notation, write $t_0=t_0(k,G)$. Let $\xi:=\xi_{\rho_\SS t}\sim  \mathrm{Normal}_d(0,\rho_\SS t\Sigma)$ and let $[\xi]$ denote its nearest integer point in $\ZZ^d$. Recall from Corollary \ref{normal_outlier} that each $w\in  \WW^{\mathrm{Normal}}_{\rho_\SS t,\delta} \cap \ZZ^d$ also belongs to the set  $\AA_{\rho_\SS t,d}=\{ x\in \mathbb{R}^d:  \|x\|_2\leq d^{-1/2}(\rho_\SS t)^{2/3}\}$. It then follows from Lemma \ref{density_comparison} that, for any $\eta>0$ and $w\in  \WW^{\mathrm{Normal}}_{\rho_\SS t,\delta} \cap \ZZ^d$,
$$\P([\xi]=w)=\int_{z\in [-1/2,1/2]^d} \phi_{\Sigma,\rho_\SS t}(w+z) dz\geq e^{-\eta d} \phi_{\Sigma,\rho_\SS t}(w)$$
when $d$ is sufficiently large. 

Now, since $w \in \WW^{\mathrm{Normal}}_{\rho_\SS t,\delta}\cap\ZZ^d$, it follows from the definition of $\WW^{\mathrm{Normal}}_{\rho_\SS t,\delta}$ in \eqref{normal_typical_set} that
\begin{align*} 
1&\geq \sum_{w\in \WW^{\mathrm{Normal}}_{\rho_\SS t,\delta}\cap \ZZ^d}\P([\xi]= w)\geq \sum_{w\in\WW^{\mathrm{Normal}}_{\rho_\SS t,\delta}\cap \ZZ^d}e^{-\eta d}\phi_{\Sigma, \rho_\SS t}( w)\geq  | \WW^{\mathrm{Normal}}_{\rho_\SS t,\delta} \cap \ZZ^d|\cdot  e^{-\eta d}\left(\frac{(1-\ep)2\pi e t_0}{(1-\delta )k}\right)^{-d/2}.
\end{align*}
Recall from \eqref{cutoff_time} that the definition of $t_0$ gives
$$e^{-\eta d}\left(\frac{(1-\ep)2\pi e t_0}{(1-\delta )k}\right)^{-d/2}=e^{-\eta d}\left(\frac{1-\ep}{1-\delta}\right)^{-d/2}\cdot |G|^{-d/(k-1)}.$$
Hence, choosing both $\delta$ and $\eta$ to be sufficiently small compared to $\ep$ ensures 
$$1\geq 2|\WW^{\mathrm{Normal}}_{\rho_\SS t,\delta} \cap \ZZ^d|\cdot |G|^{-d/(k-1)}.$$
By  definition, each point in $\WW^{\mathrm{Normal}}_{\rho_\SS t,\delta} \cap \ZZ^d$ corresponds to two points in $\WW_\SS(\delta,t)$, which then implies $|\WW_\SS(\delta,t)|=2|\WW^{\mathrm{Normal}}_{\rho_\SS t,\delta} \cap \ZZ^d|\leq |G|^{d/(k-1)}$.

\end{proof}

The next lemma relates the probability mass function of $[\xi^+_m] - [\xi^-_m]$ to the density function of a multivariate normal random variable. This serves as a preliminary result for the proof of Lemma \ref{set_comparison}.

Recall that $\hat \vp=(1,\dots,1)/(d+1)\in\mathbb{R}^d$ and $\Sigma=\mathrm{diag}(\hat \vp)-\hat \vp \hat \vp^T$. Also, recall $\AA_{m,d}=\{ x\in \mathbb{R}^d: \|x\|_2\leq d^{-1/2}m^{2/3}\}$.

\begin{lemma}\label{multinomial_density_comparison}
Suppose $1\ll d^3 \ll t$. For $m\in\mathbb{R}$, let $\xi^\pm_m$ denote two independent $\mathrm{Normal}_d\left(m\hat \vp,m\Sigma\right)$ random variables. For any $\delta>0$, if 
$$w\in \WW^{\mathrm{Normal}}_{\rho_\SS t,\delta}\cap \ZZ^d \quad \text{ and }\quad 2m\in [\rho_\SS t-2t^{7/12}, \rho_\SS t+2t^{7/12}],$$
 then for arbitrarily small $\eta>0$ (independent of $\delta$) and $ y\in\{ z\in\ZZ^d: \|z\|_1\leq 1\}$,
$$  \P([\xi^+_m]-[\xi^-_m]=w+y)\leq e^{\eta d}\phi_{\Sigma, \rho_\SS t}(w).$$
\end{lemma}
\begin{proof}
Since
$$\P\left([\xi^{\pm}_m] = x \right)=\int_{z\in [-1/2,1/2]^d} \phi_{\Sigma, m}(x-m\hat \vp+z) dz,$$ an application of Lemma \ref{density_comparison} shows that, for $x-m\hat \vp\in \AA_{m,d}$,
\beq\label{density_comparison_inequality}
e^{-\eta d}\sup_{z\in [-1/2,1/2]^d}\phi_{\Sigma, m}(x-m\hat \vp+z)\leq \P\left([\xi^{\pm}_m] = x \right)\leq e^{\eta d}\inf_{z\in [-1/2,1/2]^d}\phi_{\Sigma, m}(x-m\hat \vp+z).
\eeq
Since  $\xi^+_m-\xi^-_m\sim \mathrm{Normal}_d(0, 2m\Sigma)$, we can observe that for any $x\in \mathbb{R}^d$,
\begin{align}\label{phi_into_sum}
 \nonumber   \phi_{\Sigma, 2m}(x)&=\int_{y\in\mathbb{R}^d} \phi_{\Sigma,m}(y-m\hat \vp)\phi_{\Sigma, m}(x+y-m\hat \vp)\;dy\\
  \nonumber      &=\sum_{u\in \ZZ^d}\int_{z\in [-1/2,1/2]^d} \phi_{\Sigma, m}(u-m\hat \vp+z)\phi_{\Sigma, m}(x+u-m\hat \vp+z)\;dz\\
 \nonumber  &\geq \sum_{u\in \ZZ^d: x+u-m\hat \vp\in \AA_{m,d}}\inf_{z\in [-1/2,1/2]^d}\phi_{\Sigma, m}(x+u-m\hat \vp+z)\P([\xi^-_m]=u)\\
 \nonumber  &\geq e^{-\eta d}\sum_{u\in \ZZ^d: x+u-m\hat \vp\in\AA_{m,d}} \P\left([\xi^+_m] = x + u\right)\P([\xi^-_m]=u)\\
&=e^{-\eta d} \left( \P([\xi^+_m]-[\xi^-_m]=x)-\sum_{u\in \ZZ^d: x+u-m\hat \vp\in \AA_{m,d}^c}\P\left([\xi^+_m] = x + u\right)\P([\xi^-_m]=u)\right),
\end{align}
where the fourth line uses \eqref{density_comparison_inequality}.

To see the second term in \eqref{phi_into_sum} contributes negligibly, note that 
\begin{align*}
\sum_{u\in \ZZ^d: x+u-m\hat \vp\in \AA_{m,d}^c}\P\left([\xi^+_m] = x + u\right)\P([\xi^-_m]=u)&\leq 
\sup_{u\in \ZZ^d: x+u-m\hat \vp\in \AA_{m,d}^c}\P\left([\xi^+_m] = x + u\right)\\
&\leq  \sup_{u\in \ZZ^d: x+u-m\hat \vp\in \AA_{m,d}^c}\sup_{z\in [-1/2,1/2]^d} \phi_{\Sigma,m}(x+u-m\hat \vp+z)\\
&= \sup_{v\in \AA_{m,d}^c}\sup_{z\in [-1/2,1/2]^d} \phi_{\Sigma, m}(v+z).
\end{align*}

The same argument as that in the proof of Corollary \ref{normal_outlier} yields, for ${v\in  \AA_{m,d}^c}$ and $z\in [-1/2,1/2]^d$, we have
\beq\label{shifted_A}
(v+z)^T\Sigma^{-1}(v+z)\geq (d+1) \|v+z\|_2^2\geq d( \|v\|_2^2-2\sum_{i=1}^d |v_iz_i|)\geq  d( \|v\|_2^2-d^{1/2}\|v\|_2) \geq m^{4/3}/4
\eeq
as $1\ll d^3\ll m$. Thus, by the assumption that $2m\geq \rho_\SS t-2t^{7/12}$, we have, for any $w\in \WW^{\mathrm{Normal}}_{\rho_\SS t,\delta}$,
\begin{align*}\sup_{v\in  \AA_{m,d}^c}\sup_{z\in [-1/2,1/2]^d} \phi_{\Sigma, m}(v+z)&\leq (2\pi m)^{-d/2}(d+1)^{(d+1)/2}\exp(- m^{1/3}/8)\\
&\leq  4^d(d+1)^{1/2}e^{d/2}(1-\delta)^{-d/2} \exp(- m^{1/3}/8)\left(\frac{2\pi e \rho_\SS t}{(1-\delta)(d+1)}\right)^{-d/2}\\
&\leq e^{-2\eta d}\phi_{\Sigma,\rho_\SS t}(w),
\end{align*}
where the final inequality holds for any $\eta > 0$, and follows from the condition $1 \ll d^3 \ll m$.

Recall from Corollary \ref{normal_outlier} that  $\WW^{\mathrm{Normal}}_{\rho_\SS t,\delta} \subseteq \AA_{\rho_\SS t,d}$. For $w\in\WW^{\mathrm{Normal}}_{\rho_\SS t,\delta}\cap \ZZ^d$, taking $x=w+y$ in \eqref{density_comparison_inequality} and applying Lemma \ref{density_comparison} yields
\begin{align*}
e^{\eta d}\phi_{\Sigma,\rho_\SS t}(w)&\geq  \phi_{\Sigma, 2m}(w)\geq e^{-\eta d} \phi_{\Sigma, 2m}(w+y)\\
&\geq e^{-2\eta d} \left( \P([\xi^+_m]-[\xi^-_m]=w+y)-e^{-2\eta d}\phi_{\Sigma,\rho_\SS t}(w)\right).
\end{align*}
Rearranging gives the desired upper bound for $ \P([\xi^+_m]-[\xi^-_m]=w+y)$.

\end{proof}

The lemma below will be used in proving the upper bound on the mixing time.

\begin{lemma}\label{set_comparison}
Suppose $1 \ll k^3 \ll t_0(k,G)$ and set $d = k_\SS-1$. Fix arbitrary $\ep>0$ and set $t = (1+\ep)t_0(k,G)$. When $\delta>0$ is sufficiently small relative to $\ep$, we have
    \[
  \WW_\SS(\delta,t) \subseteq \{ w\in \ZZ^{d+1}: \P(\C_\SS=w|\typ^{\SS})\leq |G|^{-d/(k-1)}\}.
    \]
\end{lemma}

\begin{proof}

It suffices to show that when $\delta>0$ is sufficiently small relative to $\ep$, 
$$\P(\C_\SS=w|\typ^{\SS})\leq |G|^{-d/(k-1)}\quad \quad \text{ for any $w\in \WW_\SS(\delta,t)$}.$$
 For brevity, we shall always assume that $d$ is sufficiently large in the proof that follows. Recall that $\hat \cdot$ denotes the restriction to the first $d$ coordinates.

Note that
\begin{align*}
\P(\C_\SS=w,\typ^{\SS})&\leq\sum_{\ell \in  [\rho_\SS t-t^{7/12}, \rho_\SS t+t^{7/12}]} \P(M_\SS+\C_{\err}=w, N_\SS=\ell )\\
&\leq  \sum_{\ell \in [\rho_\SS t-t^{7/12}, \rho_\SS t+t^{7/12}]}\max_{y:\|y\|_1\leq 1} \P([\xi^+_{\floor{\ell/2}}]-[\xi^-_{\floor{\ell/2}}]=\hat w+\hat y)\P(  N_\SS=\ell)\\
&\leq \max_{\ell \in [\rho_\SS t-t^{7/12}, \rho_\SS t+t^{7/12}]}\max_{y:\|y\|_1\leq 1} \P([\xi^+_{\floor{\ell/2}}]-[\xi^-_{\floor{\ell/2}}]=\hat w+\hat y).
\end{align*}
Recall from Corollary \ref{normal_outlier} that $  \WW^{\mathrm{Normal}}_{\rho_\SS t,\delta} \subseteq \AA_{\rho_\SS t,d}$. By Lemma~\ref{multinomial_density_comparison} and Lemma~\ref{density_comparison}, we have, for any $\ell \in[\rho_\SS t-t^{7/12}, \rho_\SS t+t^{7/12}]$, $\hat w\in  \WW^{\mathrm{Normal}}_{\rho_\SS t,\delta}$ and $\hat y\in\{ z\in\ZZ^d: \|z\|_1\leq 1\}$,
$$\P([\xi^+_{\floor{\ell/2}}]-[\xi^-_{\floor{\ell/2}}]=\hat w+\hat y)\leq e^{\eta d}\phi_{\Sigma,\rho_\SS t}(\hat w).$$
 Plugging back in, we obtain
\begin{align*}
\P(\C_\SS=w,\typ^{\SS})&\leq e^{\eta d}\phi_{\Sigma,\rho_\SS t}(\hat w)\leq  e^{\eta d}\left(\frac{(1+\ep)2\pi e  t_0}{(1+\delta)k}\right)^{-d/2}\leq \frac{1}{2}|G|^{-d/(k-1)},
\end{align*}
where we choose $\delta$ and $\eta$ to be sufficiently small compared to $\ep$ to ensure the last inequality.
As $\P(\typ^\SS)=1-o(1)$ by \eqref{typical_event_S_prob},
$$\P(\C_\SS=w|\typ^{\SS})\leq (\P(\typ^{\SS}))^{-1} \cdot \frac{1}{2}|G|^{-d/(k-1)}\leq |G|^{-d/(k-1)}.$$

\end{proof}

The final conclusion of this section is to show that $\C_\SS(t)$ typically belongs to the set $\WW_\SS(\delta,t)$, so that the properties established earlier for $\WW_\SS(\delta,t)$ become applicable.

\begin{prop}\label{CS_entropy_concentration}
Suppose $1 \ll k^3 \ll t$. Then, for any  $\delta>0$,  
\[
\P\bigl(\C_\SS(t) \in \WW_\SS(\delta,t)\bigr) = 1 - o(1).
\]  
In particular, this holds for $t = (1\pm \ep) t_0(k,G)$ for  any $\ep>0$.

\end{prop}

\begin{proof}

 For brevity, we shall always assume that $d=k_\SS-1$ is sufficiently large in the proof that follows. From now on we suppress the index on time and write $\WW_\SS=\WW_\SS(\delta,t)$. Since $N_\SS \in [\rho_\SS t-t^{7/12}, \rho_\SS t+t^{7/12}]$ with high probability, for simplicity of notation we will write $n^{-}:=\lceil \rho_\SS t - t^{7/12}\rceil$, and $n^+=\floor{\rho_\SS t +t^{7/12}}$ throughout the proof. 
 
Let $M_\SS$ be as defined in Lemma \ref{coupling}.
Recall that $\hat{\cdot}$ denotes the restriction to the first $d$ coordinates. By the definition \eqref{WS_smallregime} we can observe that 
 \begin{align*}
\{\C_\SS\in\WW_\SS\}\supseteq \{\C_\SS \in \WW_\SS,\typ^\SS\}&=\{ M_\SS+\C_\err\in \WW_\SS, \typ^\SS\}= \{\hat M_\SS+\hat \C_{\err} \in \WW^{\mathrm{Normal}}_{\rho_\SS t,\delta}, \typ^\SS\}.
 \end{align*}
 Let $\mathcal{B}=\{\pm e_i: i\in [d]\}\cup\{0\}$ denote the 1-neighborhood of the origin, and for any set $A$ denote by $A+\mathcal{B}:=\{ w+x: w\in A, x\in \mathcal{B}\}$. Then
 \begin{align*}
\P(\C_\SS \in \WW_\SS)&\geq \P(\C_\SS\in\WW_\SS,\typ^\SS)\\
&=\P(\hat M_\SS+\hat \C_{\err} \in \WW^{\mathrm{Normal}}_{\rho_\SS t,\delta}, M_\SS=\C^+_\SS-\C^-_\SS, N_\SS\in [n_-,n_+])\\
&\geq \P(\hat M_\SS+\hat \C_{\err} \in \WW^{\mathrm{Normal}}_{\rho_\SS t,\delta},  N_\SS\in [n_-,n_+])-\P( \{M_\SS=\C^+_\SS-\C^-_\SS\}^c)\\
&\geq  \P(\hat M_\SS+\mathcal{B}\subseteq   \WW^{\mathrm{Normal}}_{\rho_\SS t,\delta}, N_\SS\in [n_-,n_+])-o(1)\\
&=\sum_{\ell\in[n_-, n_+]} \P([\xi^+_{\floor{\ell/2}}]-[\xi^-_{\floor{\ell/2}}]+\mathcal{B}\subseteq  \WW^{\mathrm{Normal}}_{\rho_\SS t,\delta})\P(N_\SS=\ell| N_\SS\in [n_-,n_+])-o(1)
\end{align*}
Hence, it suffices to show $\P([\xi^+_{\floor{\ell/2}}]-[\xi^-_{\floor{\ell/2}}]+\mathcal{B}\subseteq  \WW^{\mathrm{Normal}}_{\rho_\SS t,\delta})=1-o(1)$ for all $\ell \in[n_-,n_+]$.

Let $\xi_{2\floor{\ell/2}}:=\xi^+_{\floor{\ell/2}}-\xi^-_{\floor{\ell/2}}$. It is easy to see that $\xi_{2\floor{\ell/2}}\in [\xi^+_{\floor{\ell/2}}]-[\xi^-_{\floor{\ell/2}}]+[-1,1]^d$ and hence 
\begin{align*}
\P\left(\{[\xi^+_{\floor{\ell/2}}]-[\xi^-_{\floor{\ell/2}}]+\mathcal{B}\subseteq  \WW^{\mathrm{Normal}}_{\rho_\SS t,\delta}\}^c\right)
&\leq \P([\xi^+_{\floor{\ell/2}}]-[\xi^-_{\floor{\ell/2}}] \in ( \WW^{\mathrm{Normal}}_{\rho_\SS t,\delta})^c+\mathcal{B})\\
&\leq  \P(  \xi_{2\floor{\ell/2}}\in  ( \WW^{\mathrm{Normal}}_{\rho_\SS t,\delta})^c+[-2,2]^d).
\end{align*}
We will prove the desired result by showing, for any $\ell\in [n_-, n_+]$, 
\beq\label{error_1}
\P( \xi_{2\floor{\ell/2}} \in \AA_{\rho_\SS t,d}^c+[-2,2]^d)=o(1),
\eeq
and 
\beq\label{error_2}
\P(  \xi_{2\floor{\ell/2}} \in \AA_{\rho_\SS t,d}\cap(\WW^{\mathrm{Normal}}_{\rho_\SS t,\delta})^c +[-2,2]^d)=o(1).
\eeq

We first prove \eqref{error_1}. Recall that $\AA_{\rho_\SS t,d} = \bigl\{ x \in \mathbb{R}^d : \|x\|_2 \le d^{-1/2} (\rho_\SS t)^{2/3} \bigr\}.$
For any $x \in \AA_{\rho_\SS t,d}^c$ and $z \in [-2,2]^d$, we have
\begin{align*}
\|x+z\|_2^2 
  &\ge \|x\|_2^2 - 2 \|x\|_2 \|z\|_2 + \|z\|_2^2 
  \ge \|x\|_2^2 - 4 d^{1/2} \|x\|_2 \\
  &\ge d^{-1} (\rho_\SS t)^{4/3} - 4 (\rho_\SS t)^{2/3}\geq  0.9 \; d^{-1} (\rho_\SS t)^{4/3},
\end{align*}
where the minimum is attained at $\|x\|_2 = d^{-1/2} (\rho_\SS t)^{2/3}$, and the final inequality follows since the first term has the leading order under the assumption $1 \ll d^3 \ll t$.

As $\lfloor \ell/2 \rfloor \le n_+/2 \le 3 (\rho_\SS t)/4$, we have, for any $x \in \AA_{\rho_\SS t,d}^c$ and $z \in [-2,2]^d$,
\[
\|x+z\|_2^2 \ge 0.9 \, d^{-1} (\rho_\SS t)^{4/3} \ge (3/4)^{4/3} \, d^{-1} (\rho_\SS t)^{4/3} \ge d^{-1} (\lfloor \ell/2 \rfloor)^{4/3},
\]
which implies $x+z\in  \AA_{\lfloor \ell/2 \rfloor,d}^c$ and consequently
\[
\AA_{\rho_\SS t,d}^c + [-2,2]^d \;\subseteq\; \AA_{\lfloor \ell/2 \rfloor,d}^c.
\]
Applying Corollary~\ref{normal_outlier} with $\beta=1/2$ then proves \eqref{error_1}:
$$\P( \xi_{2\floor{\ell/2}} \in \AA_{\rho_\SS t,d}^c+[-2,2]^d)\leq \P( \xi_{2\floor{\ell/2}} \in\AA^c_{\floor{\ell/2},d})=o(1).$$

To prove \eqref{error_2}, let $\ell \in [n_-,n_+]$ and take 
$x \in \AA_{\rho_\SS t,d} \cap (\WW^{\mathrm{Normal}}_{\rho_\SS t,\delta})^c$ 
and $z \in [-2,2]^d$. We choose $\eta,\delta$ appropriately so that, by 
Lemma~\ref{density_comparison}, either
\begin{align*}
\phi_{\Sigma, 2\floor{\ell/2}}(x+z)&\leq e^{\eta d} \phi_{\Sigma,\rho_\SS t}(x)< e^{\eta d}\left(\frac{2\pi e \rho_\SS t}{(1-\delta)(d+1)}\right)^{-d/2}<\left(\frac{4\pi e \floor{\ell/2}}{(1-\delta/2)(d+1)}\right)^{-d/2},
\end{align*}
or 
\begin{align*}
\phi_{\Sigma, 2\floor{\ell/2}}(x+z)\geq e^{-\eta d} \phi_{\Sigma,\rho_\SS t}(x)>e^{-\eta d}\left(\frac{2\pi e \rho_\SS t}{(1+\delta)(d+1)}\right)^{-d/2}>  \left(\frac{4\pi e  \floor{\ell/2}}{(1+\delta/2)(d+1)}\right)^{-d/2}. 
\end{align*}
This implies that $x+z \in (\WW^{\mathrm{Normal}}_{2\lfloor \ell/2 \rfloor,\, \delta/2})^c$, 
and hence 
\[
\AA_{\rho_\SS t,d} \cap (\WW^{\mathrm{Normal}}_{\rho_\SS t,\, \delta})^c + [-2,2]^d 
   \subseteq (\WW^{\mathrm{Normal}}_{2\lfloor \ell/2 \rfloor,\, \delta/2})^c.
\]
Thus \eqref{error_2} follows from applying Lemma \ref{typical_normal},
$$\P(  \xi_{2\floor{\ell/2}} \in \AA_{\rho_\SS t,d}\cap(\WW^{\mathrm{Normal}}_{\rho_\SS t,\delta})^c +[-2,2]^d)\leq \P(  \xi_{2\floor{\ell/2}} \in (\WW^{\mathrm{Normal}}_{2\lfloor \ell/2 \rfloor,\, \delta/2})^c)=o(1).$$

The proof is therefore completed.

\end{proof}

\subsection{Regime $k\gg |G|^{2/k}(\log k)^2$ with $k\lesssim \log|G|$} 

In this regime, the auxiliary process $\C_\SS$ is no longer well-approximated by a multivariate normal distribution. Since $k$ is sufficiently large, we instead employ tensorization techniques to establish concentration of the entropy of $\C_\SS$.
More precisely, throughout this section we consider the entropy of $\C_\SS$ conditioned on $N_\SS$, treating $N_\SS$ as a fixed constant. 

\subsubsection{Preliminaries}
Recall from \eqref{CDG_process} that the discrete-time process $(Y_n)_{n\geq 0}$ is defined by
\[
Y_n = -Y_{n-1} + \delta_n=\sum_{i=1}^n (-1)^{n-i}\delta_i,
\]
with $Y_0 = 0$ and $\{\delta_j\}_{j \geq 1}$ i.i.d.\ uniform on $\{e_i : i \in [k_\SS]\}$.  
Conditioned on $N_\SS$, $\C_\SS$ has the same law as $Y_{N_\SS}$, so we will study the entropic concentration of $Y_n$.

\subsubsection{Concentration of entropy for the auxiliary process}
Note that $Y_n$ is defined in the same spirit as the Chung--Diaconis--Graham process studied in \cite{eberhard2021mixing}. We can derive entropic concentration results of $Y_n$ by adapting and refining the argument in \cite{eberhard2021mixing}.

Let $\mu^{(n)}$ denote the law of $Y_n$, and let $\{\delta_j\}_{j \geq 1}$ be i.i.d.\ uniform on $\{e_i : i \in [k_\SS]\}$.  
Without loss of generality, we may rewrite
\[
Y_n = \sum_{i=1}^n (-1)^i \delta_i,
\]
which differs from the previous definition only by a relabeling of indices. The random entropy of $Y_n$ is defined by
$$
Q_n:=Q_{n}(\delta_1,\dots,\delta_{n})=-\log \left(\mu^{(n)}\left(\sum_{i=1}^n (-1)^i\delta_i\right)\right),
$$
and the varentropy of $Y_n$ refers to $\mathrm{Var}(Q_n)$. The following proposition establishes the concentration of the random entropy $Q_n$.

\begin{prop}\label{entropy_concentration_Y}
For any $\omega>0$ and $n\in\mathbb{N}$, we have
$$
\mathbb{P}\bigl(|Q_n - \mathbb{E}[Q_n]| \geq \omega\bigr) \leq \frac{n ((\log k_\SS)^2+2)}{2\omega^2}.
$$
\end{prop}
\begin{proof}
It suffices to show $\mathrm{Var}(Q_n) \leq n ((\log k_\SS)^2+2)/2$, and then the bound follows by Chebyshev's inequality.
For simplicity of notation, we will write $Q=Q_n$ from now on.

For $1 \leq j \leq n$, let $\delta'_j$ be an independent copy of $\delta_j$ and define
\[
Q^j := Q_n(\delta_1, \dots, \delta_{j-1}, \delta'_j, \delta_{j+1}, \dots, \delta_n).
\]
By the Efron--Stein inequality (see, e.g.,  \cite{eberhard2021mixing}),
\[
\mathrm{Var}(Q) \leq \frac{1}{2} \sum_{j=1}^{n} \mathbb{E}\bigl[(Q - Q^j)^2\bigr].
\]
Thus, bounding $\mathbb{E}[(Q - Q^j)^2]$ for each $j = 1, \dots, n$ yields the desired variance estimate.

For simplicity of notation,  define $\xi^j=\sum_{i\neq j:1\leq i\leq n}(-1)^i \delta_i$ for each $1\leq j\leq n$. Let $\mu^{(n)}_j$ denote the law of $\xi^j$, i.e.,
$$\mu^{(n)}_j(x)=\P(\xi^j=x) \quad \text{ for }x\in \ZZ^{k_\SS},$$
and set  
$$f_j(\delta_1,\dots,\delta_{n})=-\log \mu^{(n)}_j(\xi^j).$$
As $f_j$ is independent from $\delta_j$, it is straightforward to observe
\begin{align}\label{tensorized_varentropy}
\nonumber \E[(Q-Q^j)^2]&\leq 2(\E[(Q-f_j(\delta_1,\dots,\delta_{n}))^2]+\E[(Q^j-f_j(\delta_1,\dots,\delta_{n})^2])\\
&=4\E[(Q-f_j(\delta_1,\dots,\delta_{n}))^2].
\end{align}
To bound $\E[(Q-f_j(\delta_1,\dots,\delta_{n}))^2]$ we consider $\P( |Q-f_j(\delta_1,\dots,\delta_{n})|>\lambda)$ for $\lambda\geq 0$. 

For a measure $\mu$ or a random variable $\xi$, we write $\mathrm{supp}(\mu)$ or $\mathrm{supp}(\xi)$ for the support.
Since for any $x\in\mathrm{supp}(\mu)$ and $y\in \mathrm{supp}(\xi^j)$ one trivially has
$\mu^{(n)}(y+(-1)^jx)\geq \mu^{(n)}_j(y)/k_\SS$,  it follows that
$$
\P(Q\geq f_j(\delta_1,\dots,\delta_{n})+\lambda)=\P(\mu^{(n)}(\xi^j+(-1)^j\delta_j) \leq e^{-\lambda}\mu^{(n)}_j(\xi^j))\leq \1_{\{\lambda\leq \log k_\SS\}}.
$$

On the other hand, note that
$$\{Q\leq  f_j(\delta_1,\dots,\delta_{n})-\lambda\}=\{\mu^{(n)}_j(\xi^j)\leq e^{-\lambda}\mu^{(n)}(\xi^j+(-1)^j\delta_j)\}.$$
We will consider the space 
$$M_\lambda=\{ (x,y)\in \mathrm{supp}(\mu)\times \mathrm{supp}(\xi^j): \mu^{(n)}_j(y)\leq e^{-\lambda}\mu^{(n)}(y+(-1)^j x)\}.$$
For a fixed $x\in \mathrm{supp}(\mu)$, define
$$M_{\lambda,x}=\{ y\in \mathrm{supp}(\xi^j): \mu^{(n)}_j(y)\leq  e^{-\lambda}\mu^{(n)}(y+(-1)^j x)\}.$$
By the definition of $M_{\lambda,x}$, we can obtain that
\begin{align*}
\P(Q\leq  f_j(\delta_1,\dots,\delta_{n})-\lambda)&= \P(\mu^{(n)}_j(\xi^j)\leq e^{-\lambda}\mu^{(n)}(\xi^j+(-1)^j\delta_j))\\
&\leq \P( (\delta_j,\xi^j)\in M_\lambda)=\sum_{x\in  \mathrm{supp}(\mu)} \P( \delta_j=x,\xi^j \in M_{\lambda,x})\\
&=\frac{1}{k_\SS} \sum_{x\in \mathrm{supp}(\mu)}\P(\xi^j \in M_{\lambda,x})=\frac{1}{k_\SS} \sum_{x\in  \mathrm{supp}(\mu)}\sum_{y\in M_{\lambda,x}} \mu^{(n)}_j(y)\\
&\leq \frac{1}{k_\SS} \sum_{x\in  \mathrm{supp}(\mu)}\sum_{y\in M_{\lambda,x}}e^{-\lambda} \mu^{(n)}(y+(-1)^jx)\leq e^{-\lambda}
\end{align*}
since for any $ x\in  \mathrm{supp}(\mu)$, $\sum_{y\in M_{\lambda,x}} \mu^{(n)}(y+(-1)^jx)\leq 1$.

It follows that, 
\begin{align*}
\E[(Q-f_j(\delta_1,\dots,\delta_{n}))^2]&=2 \int_0^\infty \lambda \; \P(  |Q-f_j(\delta_1,\dots,\delta_{n})|>\lambda) \; d\lambda \\
&\leq 2\left( \int_{0}^\infty \lambda e^{-\lambda} \; d\lambda+ \int_0^{\log k_\SS} \lambda \;d\lambda  \right)\leq (\log k_\SS)^2+2
\end{align*}
Therefore, by \eqref{tensorized_varentropy} and the Efron-Stein inequality, we have 
$$\mathrm{Var}(Q)\leq \frac{n}{2}((\log k_\SS)^2+2)$$
and the proof is completed by applying Chebyshev's inequality.
\end{proof}

\subsubsection{Entropic concentration}
We now proceed to establish entropic concentration for the auxiliary process $\C_\SS$. 
This step requires more care, as $\C_\SS$ is distributed as $Y_{N_\SS}$ with a random index $N_\SS$.

Suppose $k \gg |G|^{2/k} (\log k)^2$ and $k \lesssim \log |G|$. 
Let $1\ll \tau \ll \frac{k}{t_0(k,G)^{1/2}(\log k)}$ and define
\beq\label{typical_middle_regime}
\typ^\SS:=\{N_\SS\in  [\rho_\SS t- \tau t^{1/2}, \rho_\SS t+\tau t^{1/2}]\}=\{N_\SS\in [n_-, n_+]\}
\eeq 
where $n_-=n_-(t):=\lceil \rho_\SS t-\tau t^{7/12}\rceil, n_+=n_+(t):=\lfloor  \rho_\SS t+ \tau t^{1/2} \rfloor$.

Let $h:\mathbb{R}_+ \to \mathbb{R}_+$ denote the entropy of a rate-$1$ simple random walk on $\ZZ$, see Lemma \ref{entropy_supplementary}. For each $\ell\in  [n_-, n_+]$, define 
\beq\label{set_W_Y_minus}
\WW_{Y,\ell}^-(\delta,t):=\left\{ w\in \mathrm{supp}(Y_\ell): \P(Y_\ell=w)\geq  e^{ -k_\SS h(t /k)-\delta k}\right\},
\eeq
and 
\beq\label{set_W_Y_plus}
\WW_{Y,\ell}^+(\delta,t):=\left\{ w\in \mathrm{supp}(Y_\ell): \P(Y_\ell=w)\leq  e^{ -k_\SS h(t /k)+\delta k}\right\}.
\eeq

In addition, define
\begin{align}
\label{WSminus_middle_regime} \WW_\SS^-(\delta, t)&:=\cup_{\ell\in   [n_-, n_+]} \WW^-_{Y,\ell}(\delta,t),\\
\label{WSplus_middle_regime}\WW_\SS^+(\delta,t)&:= \WW^+_{Y,n_+}(\delta,t)\cup  \WW^+_{Y,n_+-1}(\delta,t).
\end{align}
Note that 
$
\mathrm{supp}(Y_\ell)\subseteq 
\{w\in \ZZ^{k_\SS} : \sum_{i=1}^{k_\SS} w_i = \1\{\ell \ \text{is odd}\} \},
$
so the union in~\eqref{WSplus_middle_regime} is disjoint.

The following lemma is analogous to Lemma~\ref{WS_size} and is intended to show that, 
prior to the proposed mixing time $t_0(k,G)$, the process $\C_\SS$ 
is confined to a set of relatively small size.

\begin{lemma}\label{WS_size_middle_regime}
Assume that $k \gg |G|^{2/k} (\log k)^2$ and $k \lesssim \log |G|$. Fix $\ep>0$ and set $t = (1-\ep)\,t_0(k,G)$. 
For $\delta>0$ sufficiently small relative to $\ep$, we have
\[
|\WW^-_\SS(\delta,t)| \;\leq\; |G|^{\rho_\SS},
\]
where $\WW^-_\SS(\delta,t)$ is defined as in~\eqref{WSminus_middle_regime}. 

\end{lemma}

\begin{proof}
We begin by considering the regime  $k \gg |G|^{2/k} (\log k)^2$ with $k\ll \log|G|$. 
By the definition \eqref{cutoff_time} of $t_0=t_0(k,G)$ and the asymptotics of $h(\cdot)$ in Lemma \ref{entropy_supplementary}, in this regime we have 
$$ h(t_0/k)= \frac{1}{k-1}\log|G|+O( (t_0/k)^{-1/4})=\frac{1}{k}\log|G|+o(1).$$
For $t=(1-\ep)t_0$, again by the asymptotics of $h(\cdot)$, there exists some constant $\alpha>0$ such that 
$$h(t/k)=h((1-\ep)t_0/k)=h(t_0/k)+\frac{1}{2}\log(2\pi e(1-\ep))+O((t/k)^{-1/4})\leq \frac{1}{k}\log|G|-\alpha\ep$$
when $k$ is sufficiently large.

Choose $\delta\leq \alpha \rho_\SS \ep/2$, so that by the definition \eqref{set_W_Y_minus}, for any $\ell$,
\begin{align*}
1&\geq \sum_{w\in  \WW^-_{Y,\ell}(\delta,t)}\P(Y_\ell=w)\geq e^{ -k_\SS h(t/k)-\delta k}\cdot |\WW^-_{Y,\ell}(\delta,t)|\geq |G|^{-\rho_\SS}e^{ \alpha\ep k_\SS -\delta k} \cdot |\WW^-_{Y,\ell}(\delta,t)|\\
&\geq |G|^{-\rho_\SS}e^{ \alpha\ep k_\SS /2}\cdot |\WW^-_{Y,\ell}(\delta,t)|.
\end{align*}
 This implies $|\WW_{Y,\ell}(\delta,t)|\leq |G|^{\rho_\SS}e^{-\alpha \ep k_\SS/2}$ and therefore
$$|\WW^-_\SS(\delta,t)|=|\cup_{\ell \in [n_-,n_+]} \WW^-_{Y,\ell}(\delta,t)|\leq \sum_{\ell \in [n_-,n_+]} |\WW^-_{Y,\ell}(\delta,t)|\leq 4\tau t^{1/2}|G|^{\rho_\SS}e^{-\alpha \ep k_\SS/2}\leq |G|^{\rho_\SS}$$
as $\tau t^{1/2}\ll k\ll e^{\alpha \ep k_\SS/2}$.

The proof for the regime $k \asymp \log |G|$ follows analogously. In this case, we have exactly $h(t_0/k) = \frac{1}{k}\log |G|$. By the mean value theorem and Lemma~\ref{entropy_supplementary}, there exists a constant $\alpha' > 0$ such that 
\[
h(t/k) \leq h(t_0/k) - \alpha'\varepsilon 
= \frac{1}{k}\log |G| - \alpha'\varepsilon .
\]
Choosing $\delta \leq \tfrac{1}{2}\alpha'\rho_\SS\ep$, 
the remainder of the argument follows identically as before.

\end{proof}

Next, we establish an analogue of Lemma~\ref{set_comparison} 
for the regime $k \gg |G|^{2/k} (\log k)^2$ and $k \lesssim \log |G|$ under consideration.

\begin{lemma}\label{set_comparison_Y}
Assume that $k \gg |G|^{2/k} (\log k)^2$ and $k \lesssim \log |G|$.
Let $\ep>0$ be fixed and set $t = (1+\ep)t_0(k,G)$. For $\delta$ sufficiently small relative to $\ep$, we have
$$\WW^+_\SS(\delta,t) \subseteq \{w\in \ZZ^{k_\SS}: \P(\C_\SS=w|\typ^\SS)\leq |G|^{-\rho_\SS}\}.$$

\end{lemma}
\begin{proof}

Our goal is to show that any $w\in \WW^+_\SS(\delta,t)$ satisfies $\P(\C_\SS=w|\typ^\SS)\leq |G|^{-\rho_\SS}$. Due to the law of $\C_\SS$ given in Proposition \ref{auxiliary_law}, it suffices to consider $w \in \ZZ^{k_\SS}$ such that either $\sum_{i=1}^{k_\SS} w_i = 0$ or $ \sum_{i=1}^{k_\SS} w_i = 1$.

Without loss of generality, we first consider the case $\sum_{i=1}^{k_\SS} w_i = 0$, and
the case $\sum_{i=1}^{k_\SS} w_i = 1$ can be treated in the same way.  
For $w \in \ZZ^{k_\SS}$ with $\sum_{i=1}^{k_\SS} w_i = 0$, we have
\[
\P(\C_\SS = w| \typ^\SS)
= \sum_{\ell \in [n_-, n_+] \cap 2\ZZ} \P(Y_\ell = w)\, \P(N_\SS = \ell| \typ^\SS)
\le \max_{\ell \in [n_-, n_+] \cap 2\ZZ} \P(Y_\ell = w).
\]
Let $n_* \in \{\, n_+,\, n_+-1 \,\} \cap 2\ZZ$ denote the unique even number in this set.  
Observe that for any $\ell \in [n_-, n_+] \cap 2\ZZ$, the difference $n_* - \ell$ is even.  
By choosing the steps between $\ell$ and $n_*$ so that they cancel pairwise, we obtain
\[
\P(Y_{n_*} = w) \ge \P(Y_\ell = w) \left(\frac{1}{k_\SS}\right)^{n_* - \ell}.
\]
 Hence, by the definition of $\WW^+_\SS(\delta,t) $,
\begin{align*}
\P(\C_\SS=w|\typ^\SS)&\leq  \max_{\ell \in [n_-, n_+] \cap 2\ZZ} \P(Y_\ell = w)\leq  \max_{\ell \in [n_-, n_+] \cap 2\ZZ}(k_\SS)^{n_*-\ell}\P(Y_{n_*}=w)\\
&\leq e^{\delta k}e^{-k_\SS h(t/k)+\delta k}= e^{-k_\SS h(t/k)+2\delta k},
\end{align*}
where the last line follows from noting $ (k_\SS)^{n_*-\ell}\leq e^{\delta k}$ from the assumption $\tau t^{1/2}(\log k)\ll k$.

As discussed in the proof of Lemma \ref{WS_size_middle_regime}, we have $h(t_0/k)=\frac{1}{k}\log|G|+o(1)$. Furthermore, for $t=(1+\ep)t_0$, by analogous arguments to the proof of Lemma \ref{WS_size_middle_regime}
we can obtain
$$h(t/k)\geq h(t_0/k)+\alpha\ep$$ 
for some $\alpha>0$. Thus, when $\delta$ is sufficiently small compared to $\ep$, the desired conclusion follows
$$\P(\C_\SS=w|\typ^\SS)\leq e^{-\rho_\SS\log|G|-\alpha\ep k_\SS+2\delta k}\leq |G|^{-\rho_\SS}.$$

The case $\sum_{i=1}^{k_\SS} w_i = 1$ is handled in the same way, with $2\ZZ$ replaced by $2\ZZ+1$ in the index set.
\end{proof}

We are now ready to prove entropic concentration in Regime (ii).

\begin{prop}\label{CS_entropy_concentration_middleregime}
Assume $k \gg |G|^{2/k} (\log k)^2$ and $k \lesssim \log |G|$.  
Let $\varepsilon>0$ be fixed and  $\delta$ is chosen sufficiently small relative to $\ep$. We have:
\begin{enumerate}[label=(\roman*)]
    \item If $t = (1-\ep)t_0(k,G)$, then
$$
\P(\C_\SS(t) \in \WW_\SS^-(\delta,t)) = 1 - o(1).
$$
    
    \item If $t= (1+\ep)t_0(k,G)$, then
$$
\P(\{\C_\SS(t) \in \WW_\SS^+(\delta,t)\}\cap \typ^{\SS}) = 1 - o(1).
$$
\end{enumerate}

\end{prop}
\begin{proof}
Recall that $\mu^{(\ell)}$ denotes the law of $Y_\ell$ and $Q_\ell=-\log \mu^{(\ell)}(Y_\ell)$ denotes the random entropy of $Y_\ell$. We begin the proof of (i) by showing that for $t=(1-\ep)t_0(k,G)$, $\delta>0$ sufficiently small relative to $\ep$, 
\beq
 \P(Y_\ell \notin \WW^-_{Y,\ell}(\delta,t))=o(1) \quad \text{ for any }\ell\in [n_-,n_+].
 \eeq
 From now on we drop the indices and write $\WW^{\pm}_\SS=\WW^{\pm}_\SS(\delta,t)$, $\WW^{\pm}_{Y,\ell}=\WW^{\pm}_{Y,\ell}(\delta,t)$.
 
As we set $t = (1-\ep)t_0(k,G)\asymp k|G|^{2/k}$, the regime under consideration ensures that  $k^2 \gg t (\log k)^2$, which implies $k \gg \log t$.
Moreover, since $\ell = (1+o(1))\,\rho_\SS t$, it follows from 
Corollary~\ref{Y_entropy} together with Lemma~\ref{entropy_supplementary} that

\[
\E[Q_\ell]=H(Y_\ell) 
= k_\SS \cdot h\!\left(\frac{\ell}{k_\SS}\right) + o(k)
= k_\SS \cdot h\!\left(\frac{\rho_\SS t}{k_\SS}\right) + o(k) 
= k_\SS \cdot h(t/k) + o(k).
\]
By Lemma~\ref{entropy_concentration_Y} and the assumption that  $k^2 \gg t (\log k)^2$,
\begin{align}\label{error_Y}
\nonumber \P(Y_\ell \notin \WW^-_{Y,\ell})&=\P(-\log \mu^{(\ell)}(Y_\ell)\geq   k_\SS h(t /k)+\delta k)\\
\nonumber &\leq \P( Q_\ell-\E[Q_\ell]\geq  k_\SS h(t /k)-\E[Q_\ell]+\delta k)\leq  \P( |Q_\ell-\E[Q_\ell]|\geq  \delta k/2)\\
&\leq \frac{4\ell(\log k)^2}{(\delta k)^2}=o(1).
\end{align}

By the definition $\WW_\SS^-=\cup_{\ell\in   [n_-, n_+]}  \WW^-_{Y,\ell}$, it is then immediate that
\begin{align*}
\P(\C_\SS \notin \WW_\SS^-)&\leq \sum_{\ell \in [n_-,n_+]} \P( Y_{N_\SS} \notin \WW_\SS^-,N_\SS=\ell)+\P(N_\SS\notin [n_-,n_+])\\
&=\sum_{\ell \in [n_-,n_+]}\P( Y_\ell\notin \WW_\SS^- )\P(N_\SS=\ell)+o(1)\\
&\leq \sum_{\ell \in [n_-,n_+]} \P(Y_\ell \notin \WW_{Y,\ell})\P(N_\SS=\ell)+o(1)=o(1).
\end{align*}
This proves part (i). 

\medskip
Next, we turn to the proof of (ii). Let $t\geq (1+\ep)t_0(k,G)$, and observe
$$
\P(\{\C_\SS \in \WW_\SS^+\}\cap \typ^{\SS})=\sum_{\ell\in [n_-,n_+]}\P( Y_\ell \in \WW_\SS^+)\P(N_\SS=\ell) \geq \min_{\ell\in [n_-,n_+]}\P( Y_\ell \in \WW_\SS^+)\P(N_\SS\in [n_-,n_+]).
$$
Hence it remains to show, for all $\ell \in [n_-,n_+]$,
$$\P( Y_\ell \in \WW_\SS^+)=1-o(1).$$
Let $n_\ell\in \{n_+,n_+-1\}$ denote the unique number satisfying $(n_\ell-\ell) \mod 2 \equiv 0$ (and note that $n_\ell \geq \ell$), so that 
 $$\{Y_\ell\in  \WW_\SS^+\}=\{ Y_\ell\in  \WW^+_{Y,n_\ell}\}.$$
Take $\omega \gg 1$ to be diverging arbitrarily slowly. We can note that trivially $Y_\ell\in \mathrm{supp}(\mu^{(n_\ell)})$ and
\begin{align*}
\P( Y_\ell \in \WW_\SS^+)&=\P( Y_\ell\in \WW_{Y,n_\ell})= \P(\mu^{(n_\ell)}(Y_\ell)\leq e^{-k_\SS h(t/k)+\delta k})\\
&\geq \P( \mu^{(n_\ell)}(Y_\ell)\leq e^{\omega} \mu^{(\ell)}(Y_\ell)\leq e^{-k_\SS h(t/k)+\delta k}).
\end{align*}
Thus,
$$\P( Y_\ell \notin \WW_\SS^+)\leq \P(  \mu^{(n_\ell)}(Y_\ell)>  e^{\omega} \mu^{(\ell)}(Y_\ell))+\P(e^{\omega} \mu^{(\ell)}(Y_\ell)> e^{-k_\SS h(t/k)+\delta k}).$$

To address the first term, let $M_{\omega,\ell}=\{y\in \mathrm{supp}(Y_\ell): \mu^{(n_\ell)}(y)\geq e^{\omega}\mu^{(\ell)}(y)\}$, and we have
\begin{align*}
\P(  \mu^{(n_\ell)}(Y_\ell)\geq  e^{\omega} \mu^{(\ell)}(Y_\ell))&=\P(Y_\ell\in M_{\omega,\ell})=\sum_{y\in M_{\omega,\ell}}\mu^{(\ell)}(y)\\
&\leq \sum_{y\in M_{\omega,\ell}}e^{-\omega}\mu^{(n_\ell)}(y)\leq e^{-\omega}=o(1).
\end{align*}
Repeating the argument that  leads to \eqref{error_Y}, we can show that  
$\P(e^{\omega} \mu^{(\ell)}(Y_\ell)> e^{-k_\SS h(t/k)+\delta k})=o(1)$ when $1\ll \omega\ll k$. Therefore, we have shown that for all $\ell \in [n_-,n_+]$,
$$\P( Y_\ell \notin \WW_\SS^+)=o(1),$$
and this implies $\P(\{\C_\SS \in \WW_\SS^+\}\cap \typ^{\SS})=1-o(1)$.

\end{proof}

\section{Mixing times in Regime $k\lesssim \log|G|$}\label{sec:lower_regime}
\subsection{Lower bound on mixing time}

In this section, we prove the lower bounds on mixing times for the regimes: 
\begin{enumerate}[label=(\roman*)]
\item $k\gg1$ with $k^2\ll |G|^{2/k}$;
\item $k\gg |G|^{2/k}(\log k)^2$ with $k\lesssim \log|G|$. 
\end{enumerate}
We collect the definitions used before in the following. Recall that $\rho_\SS = k_\SS/k$ and $\rho_\RR = k_\RR/k$. In Regime~(i), we define $\WW^-_\SS(\delta,t)$ as in \eqref{WS_smallregime}, while in Regime~(ii) we use the definition given in \eqref{WSminus_middle_regime}. Let $\hat w$ denote the first $k_\SS-1$ coordinates of $w\in \ZZ^{k_\SS}$.

\begin{definition}\label{entropic_set_lb}
Let $t_0=t_0(k,G)$ as in~\eqref{cutoff_time}, and take parameters $1 \ll \omega \ll k$, $1 \ll \tau \ll \tfrac{k}{t_0^{1/2}\log k}$, and $\delta>0$.

\medskip
\noindent  \textbf{\emph{Regime (i):}} Define
\begin{align*}
\WW^-_\SS(\delta, t) 
&:= \Bigl\{ w \in \ZZ^{k_\SS} : \sum_{i=1}^{k_\SS} w_i \in\{ 0,1\},\; 
\hat w \in \WW^{\mathrm{Normal}}_{\rho_\SS t,\delta} \cap \ZZ^{k_\SS-1} \Bigr\},\\
 \WW^-_\RR(t) &:= \{ w \in \ZZ^{k_\RR} : \P(\C_\RR(t) = w) \ge e^{\omega} |G|^{-k_\RR/(k-1)} \}.
\end{align*}

\medskip
\noindent \textbf{\emph{Regime (ii):}} Let $n_- = \lceil \rho_\SS t - \tau t^{1/2} \rceil$ and $n_+ = \lfloor \rho_\SS t + \tau t^{1/2} \rfloor$, and define
\begin{align*}
 \WW^-_\SS(\delta, t) 
&:= \bigcup_{\ell \in [n_-, n_+]} \Bigl\{ w \in \ZZ^{k_\SS} : \P(Y_\ell = w) \ge e^{- k_\SS h(t/k) - \delta k} \Bigr\},\\
 \WW^-_\RR(t) 
&:= \{ w \in \ZZ^{k_\RR} : \P(\C_\RR(t) = w) \ge e^{\omega} |G|^{-\rho_\RR} \}.
\end{align*}
\end{definition}

\begin{prop}\label{lb_mixing}
Let $k$ fall into one of the following regimes: 
(i) $k \gg 1$ with $k^2 \ll |G|^{2/k}$, or 
(ii) $k \gg |G|^{2/k} (\log k)^2$ with $k \lesssim \log |G|$. For a given generating set $\mathcal{S} = \{ z_a^{\pm 1} : a \in [k] \} \subseteq G$, any $\varepsilon > 0$, 
and $t \leq (1-\varepsilon) t_0(k,G)$, the rate-1 continuous time random walk $X=(X(t))_{t\geq 0}$ on $\Cay(G,\mathcal{S})$ satisfies
\[
\| \P_{\mathcal{S}}(X(t)= \cdot) - \pi \|_{\mathrm{TV}} = 1 - o(1),
\]
where $\P_{\mathcal{S}}$ denotes the law of $X$ started at $X(0) = \id$ driven by ${\mathcal{S}}$.

\end{prop}
\begin{proof}
Let $\varepsilon>0$ be fixed and set $t = (1-\varepsilon)\,t_0(k,G)$. 
Since the total variation distance is non-increasing, it suffices to prove the result for this choice of $t$. 
For simplicity of notation, we omit the indices and write $\P = \P_{\mathcal{S}}$ and $X = X(t)$.
Let $k_\SS$ be the number of generators in $\{z_a: a\in[k]\}$ that are reflections and we can write
$$z_a=
\begin{cases}
sr^{u_a} & \text{ for }1\leq a\leq k_\SS\\
r^{u_a}& \text{ for } k_\SS<a\leq k
\end{cases}
$$
for $(u_a)_{a\in [k]}\subseteq [n]$. 

Let $\delta>0$ be sufficiently small relative to $\ep$, and let $\WW^-_\RR=\WW^-_\RR(t),\WW^-_\SS=\WW^-_\SS(\delta,t)$ be defined as in Definition \ref{entropic_set_lb}.
Consider the set 
\begin{align*}
E=\{ x\in G: \exists w\in \WW_\RR^-\times \WW^-_\SS \text{ such that }x=r^{\sum_{a\in [k]} w_au_a} \text{ or } sr^{\sum_{a\in [k]} w_au_a}\}.
\end{align*}
It then follows from Propositions \ref{CS_entropy_concentration} and \ref{CS_entropy_concentration_middleregime} that, when $t = (1-\ep) t_0(k,G)$,
\beq\label{lower_bound_err}
\P(X \in E ) \geq \P(\C \in \WW_\RR^- \times \WW_\SS^- )=\P(\C_\RR\in \WW_\RR^-)\P(\C_\SS\in \WW^-_\SS) = 1 - o(1).
\eeq

In Regime (i),  the definition of $\WW^-_\RR$ implies 
$$1\geq \sum_{w\in \WW^-_\RR}\P(\C_\RR=w)\geq e^{\omega}|G|^{-k_\RR/(k-1)}|\WW^-_\RR|$$
and thus $|\WW^-_\RR|\leq e^{-\omega}|G|^{k_\RR/(k-1)}$, while Lemma \ref{WS_size} shows that $|\WW^-_\SS|\leq |G|^{(k_\SS-1)/(k-1)}$. In Regime (ii), it follows from the same reasoning that $|\WW^-_\RR|\leq e^{-\omega}|G|^{\rho_\RR}$ and Lemma \ref{WS_size_middle_regime} implies $|\WW^-_\SS|\leq |G|^{\rho_\SS}$. Therefore, in both regimes, we have $| \WW_\RR^-\times \WW^-_\SS|\leq e^{-\omega}|G|$. 

Due to the correspondence between $ \WW_\RR^-\times \WW^-_\SS$ and $E$, 
$$|E|\leq 2| \WW_\RR^-\times \WW^-_\SS|\leq  2e^{-\omega}|G|.$$
Combined with \eqref{lower_bound_err}, we have shown
$$\| \P(X(t)=\cdot) - \pi \|_{\mathrm{TV}} \geq \P(X\in E)-\pi(E)\geq 1-o(1)-2e^{-\omega}=1-o(1).$$

\end{proof}

\subsection{Upper bound on mixing time}
Recall the two regimes: (i) $k \gg 1$ with $k^2 \ll |G|^{2/k}$, and (ii) $k \gg |G|^{2/k} (\log k)^2$ with $k \lesssim \log|G|$. Throughout this section, we will refer to them as Regime~(i) and Regime~(ii).

We again collect the definitions of the special events and sets introduced in the previous discussion. Recall the definition of $M_\SS$ from Lemma \ref{coupling}. Let $\hat w$ denote the first $k_\SS-1$ coordinates of $w\in \ZZ^{k_\SS}$.
\begin{definition}\label{typical_event}

Let $t_0=t_0(k,G)$ as in~\eqref{cutoff_time}, and take parameters $1 \ll \omega \ll k$, $1 \ll \tau \ll \tfrac{k}{t_0^{1/2}\log k}$, and $\delta>0$. 

\medskip
\noindent  \textbf{\emph{Regime (i):}} Define
\begin{align*}
\WW^+_\SS(\delta,t)&:=\{ w\in \ZZ^{k_\SS}: \sum_{i=1}^{k_\SS}w_i\in \{0,1\}, \hat w\in \WW^{\mathrm{Normal}}_{\rho_\SS t,\delta}\cap \ZZ^{k_\SS-1}\},\\
\WW^+_\RR(t)&:=\{ w\in \ZZ^{k_\RR}: \P(\C_\RR(t)=w)\leq e^{-\omega}|G|^{-k_\RR/(k-1)}\},\\
\typ^{\SS}&:=\{ M_\SS=  \C^+_\SS- \C^-_\SS\}\cap \{N_\SS\in  [\rho_\SS t-t^{7/12}, \rho_\SS t+t^{7/12}]\}. 
\end{align*}

\medskip
\noindent \textbf{\emph{Regime (ii):}} Let $n_- = \lceil \rho_\SS t - \tau t^{1/2} \rceil$ and $n_+ = \lfloor \rho_\SS t + \tau t^{1/2} \rfloor$, and define
\begin{align*}
\WW_\SS^+(\delta,t)&:= \left\{ w\in \mathbb{\ZZ}^{k_\SS}:  \P(Y_{n_+}=w)\leq e^{-k_\SS h(t/k)+\delta k}\right\} \\
 &\quad\quad\quad\quad \quad \cup  \left\{ w\in \mathbb{\ZZ}^{k_\SS}:  \P(Y_{n_+-1}=w)\leq e^{-k_\SS h(t/k)+\delta k}\right\}\\
\WW^+_\RR(t)&:=\{ w\in \ZZ^{k_\RR}: \P(\C_\RR(t)=w)\leq e^{-\omega}|G|^{-\rho_\RR}\}\\
 \typ^{\SS}&:=\{N_\SS\in  [\rho_\SS t-\tau t^{1/2}, \rho_\SS t+\tau t^{1/2}]\}.
 \end{align*}

\noindent
In both Regime (i) and (ii), let $\WW^+(\delta,t):=\WW_\RR^+(t)\times \WW_\SS^+(\delta,t)$. 
\end{definition}

\medskip
We now follow the entropic methodology described in Section \ref{entropic_framework} to prove the upper bound on the mixing time. Let $X'$ be an independent copy of $X$. Similarly, let $\C'=(\C'_\RR,\C'_\SS)$ denote the auxiliary process of $X'$, which is an independent copy of the auxiliary process $\C$ of $X$. For ease of notation, we omit the time index $t$ below. 
Let $\typ^{\SS}(X), \typ^\SS(X')$ be defined as in Definition \ref{typical_event}, for the evolution of $X$ and $X'$, respectively. 

Define
\beq\label{typical_event_X}
\typ:= \{\typ^{\SS}(X), \C\in \WW^+\}\cap \{\typ^\SS(X'),\C'\in\WW^+\}
\eeq
It follows from the same arguments as Lemma \ref{TV_entropic} that 
\beq\label{tv_conditional_ub}
\| \P_\mathcal{S}(X= \cdot) - \pi \|_{\mathrm{TV}} \leq \| \P_\mathcal{S}(X= \cdot \mid \C\in\WW^+,\typ^\SS(X) ) - \pi \|_{\mathrm{TV}}+\P(\{\C\in \WW^+,\typ^\SS(X)\}^c)
\eeq
and 
$$4\| \P_\mathcal{S}(X = \cdot \mid  \C\in\WW^+,\typ^\SS(X)) - \pi \|_{\mathrm{TV}}^2\leq |G|\cdot \P_\mathcal{S}(X=X'|\typ)-1.$$
Note that when $\mathcal{S}$ is a random set of generators, $\| \P_\mathcal{S}(X = \cdot) - \pi \|_{\mathrm{TV}}$ is a random variable that is measurable with respect to the $\sigma$-field generated by the choice of $\mathcal{S}$. In what comes later in our arguments, we will take the expectation over the choices of $\mathcal{S}$ and work with
$$\P(X=X'|\typ):=\E[\P_\mathcal{S}(X=X'|\typ)].$$

As suggested by Lemma \ref{TV_entropic}, in order to control the total variation distance to stationarity we will upper bound  
\begin{align}\label{tv_decomp}
\nonumber D(t)&=|G|\cdot \P(X=X'|\typ)-1\\
&\leq |G|\cdot \P(X=X',\C\neq \C' |\typ)+|G|\cdot \P(\C=\C' | \typ)-1.
\end{align}
where we also average over the choice of $\mathcal{S}$. 

In order to prove the upper bound on the mixing time, we first control the probability of $\typ$, as established in Lemma~\ref{typical_event_prob}. We then address the two terms appearing in~\eqref{tv_decomp} through Lemmas~\ref{nonequalC} and~\ref{equalC}, respectively.

\begin{lemma}\label{typical_event_prob}
Let $k$ belong to Regime~(i) or~(ii), fix $\ep>0$, and set $t = (1+\ep) t_0(k,G)$. 
Choose $\delta>0$ sufficiently small relative to $\ep$ and define 
$\WW^+ = \WW^+(\delta, t)$ as in Definition~\ref{typical_event}. Then
\[
\P(\C(t) \in \WW^+, \typ^\SS) = 1 - o(1),
\]
and consequently $\P(\typ) = 1 - o(1)$, where $\typ$ is defined in~\eqref{typical_event_X}.
\end{lemma}

\begin{proof}
By the independence of $\C_\RR$ from $\C_\SS$ and $\FF_\SS$, as stated in Proposition~\ref{auxiliary_law}, we have 
\[
\P(\C \in \WW^+, \typ^\SS) 
  = \P(\C_\RR \in \WW^+_\RR)\, \P(\{\C_\SS \in \WW^+_\SS\}\cap \typ^\SS). 
\]
Lemma~\ref{CR_concentration} implies that $\P(\C_\RR \in \WW^+_\RR) = 1 - o(1)$, 
while Propositions~\ref{CS_entropy_concentration} and~\ref{CS_entropy_concentration_middleregime} 
together yield 
\[
\P\bigl(\{\C_\SS \in \WW^+_\SS\} \cap \typ^\SS \bigr) = 1 - o(1)
\]
for $k$ in Regime (i) or (ii).
\end{proof}

Linear combinations of independent uniform random variables in an abelian group are themselves uniform on their support. As preparation for the proof we present the following lemma, which is a restatement of known results in   \cite{hermon2021cutoff}.
\begin{lemma}[Lemma 2.11  \cite{hermon2021cutoff}]\label{gcd_unif}
Let $k\in\mathbb{N}$. Let $H$ be an Abelian group and $U_1,\dots,U_k\overset{iid}{\sim} \Unif(H)$. For $v=(v_1,\dots, v_k)\in \ZZ^k$, write $U=(U_1,\dots,U_k)$ and define $v\cdot U:=\sum_{i=1}^k v_iU_i$. We have
$$v \cdot U\sim \Unif (\gamma H) \quad \text{ where } \gamma=\gcd(v_1,\dots,v_k,|H|).$$

\end{lemma}

\begin{lemma}\label{nonequalC}
Let $k$ belong to either Regime~(i) or~(ii). For any $\ep>0$ and $t\geq  (1+\ep)t_0(k,G)$, we have 
\begin{align*}
|G|\cdot \P(X=X',\C\neq \C'| \typ)&=1+o(1).
\end{align*}
\end{lemma}
\begin{proof}
Write $V:=\C-\C'$ and let $U=(U_1,\dots,U_k)$ where $U_a$ is defined as in \eqref{gen_rep}, which are i.i.d. uniform elements of $\ZZ_n$. Recall from Definition~\ref{filtration} the filtration $\HH$, which encodes the sequence information without specifying the identities of $(U_a)_{a\in [k]}$. Under $\HH$, $\C$ and $\C'$ are measurable, and hence $V$ is known.

Recall from Definition \ref{filtration} the filtration $\HH$, which essentially records the sequence information without the identities of $(U_a)_{a\in [k]}$, under which $\C,\C'$ are measurable and thus $V$ is known. 

When $V\neq 0$, as $n$ is a prime, Lemma \ref{gcd_unif} implies that $V\cdot U\sim \Unif (\ZZ_n)$. It then follows from Proposition \ref{simplified_X} that
$$
\{X=X'\}=\{s^{(N_\SS-N'_\SS \mod 2)} r^{ \sum_{a\in [k]} (C_a-C'_a) U_a}=\id\},
$$
and hence
\begin{align*}
\P(X=X', V\neq 0|\HH)&=\1\{V\neq 0\}\P(s^{(N_\SS-N'_\SS) \mod 2} r^{V\cdot U}=\id|\HH)\\
&=\1\{V\neq 0, (N_\SS-N'_\SS) \mod 2=0\}\P(V\cdot U=0|V)\\
&=\1\{V\neq 0, (N_\SS-N'_\SS) \mod 2=0\}\cdot \frac{1}{n}.
\end{align*}
Note that $N_\SS,N'_\SS$ are independent Poisson random variables with mean $\rho_\SS t$. Hence,
\begin{align*}
\P(X=X', \C\neq \C')&=\E[\P(X=X', V\neq 0|\HH)]\\
&\leq \frac{1}{n} \P((N_\SS-N'_\SS) \mod 2=0)=\frac{1}{n}\P((N_\SS+N'_\SS) \mod 2=0)\\
&=\frac{1}{n}\P( \mathrm{Poisson}(2\rho_\SS t) \mod 2=0)=\frac{1}{2n}(1+e^{-4\rho_\SS t}),
\end{align*}
where the last equality is a standard result that the probability for a $\mathrm{Poisson}(\lambda)$ random variable to be even is $\frac{1}{2}(1+e^{-2\lambda})$. 
Therefore, as $\P(\typ)=1-o(1)$ by Lemma \ref{typical_event_prob},
$$\P(X=X',\C\neq \C'|\typ)\leq \frac{ \P(X=X',\C\neq \C')}{\P(\typ)}\leq \frac{ |G|^{-1} (1+e^{-4\rho_\SS t})}{\P(\typ)}\leq |G|^{-1}(1+o(1)).$$
\end{proof}

 \begin{lemma}\label{equalC}
Let $k$ belong to either Regime~(i) or~(ii).   For any $\ep > 0$ and $t = (1+\ep)t_0(k,G)$,
\[
\P(\C = \C'| \typ) = o(|G|^{-1}).
\]
\end{lemma}
\begin{proof}
Let $\delta > 0$ be chosen sufficiently small relative to $\ep$, and define 
$\WW^+_\SS := \WW^+_\SS(\delta, t)$ and $\WW^+_\RR := \WW^+_\RR(t)$ as in 
Definition~\ref{typical_event}. Recall also that 
$
\WW^+ = \WW^+_\SS \times \WW^+_\RR.
$

For $w=(w_\RR,w_\SS)\in \WW^+_\RR\times \WW^+_\SS$, we can observe that
\begin{align*}
\P(\C=w|\typ)&=\P(\C=w| \C\in \WW^+, \typ^{\SS})= \frac{\P(\C_\RR=w_\RR,\C_\SS=w_\SS,\typ^\SS)}{ \P( \C \in \WW^+,\typ^\SS)}\\
&=\frac{\P(\C_\RR=w_\RR)\P(\C_\SS=w_\SS|\typ^\SS)\P(\typ^\SS)}{ \P( \C \in \WW^+,\typ^\SS)}.
\end{align*}
In Regime~(i), Definition~\ref{typical_event} yields 
$
\P(\C_\RR = w_\RR) \le e^{-\omega} |G|^{-k_\RR/(k-1)},
$
while Lemma~\ref{set_comparison} gives 
$
\P(\C_\SS = w_\SS \mid \typ^\SS) \le |G|^{-(k_\SS-1)/(k-1)}.
$
Similarly, in Regime~(ii) we have $
\P(\C_\RR = w_\RR) \le e^{-\omega} |G|^{-\rho_\RR}$, and, by Lemma~\ref{set_comparison_Y},
$
\P(\C_\SS = w_\SS \mid \typ^\SS) \le |G|^{-\rho_\SS}.
$

Hence, for $k$ in Regime~(i) or~(ii) and $w\in \WW^+_\RR \times \WW^+_\SS$,  
\begin{align*}
\P(\C = w| \typ) 
&\le e^{-\omega} |G|^{-1} \cdot 
   \frac{\P(\typ^\SS)}{\P(\C \in \WW^+, \typ^\SS)} \\
&\le 2 e^{-\omega} |G|^{-1},
\end{align*}
since $\P(\typ^\SS) = 1 - o(1)$ by~\eqref{typical_event_S_prob} and  
$\P(\C \in \WW^+, \typ^\SS) = 1 - o(1)$ by Lemma~\ref{typical_event_prob}.

The conclusion then follows
\begin{align*}
\P(\C=\C'|\typ)&=\sum_{w\in \WW^+} \P(\C=w|\typ)^2\\
&\leq 2e^{-\omega}|G|^{-1}\sum_{w\in \WW^+}\P(\C=w|\typ)\leq 2e^{-\omega}|G|^{-1}=o(|G|^{-1}).
\end{align*}

\end{proof}

Combining the above estimates, we obtain the desired upper bound on the mixing time.

\begin{prop}\label{ub_mixing}
Let $k$ be in Regime (i) or (ii). Let $\mathcal{S} = \{ Z_a^{\pm 1} : a \in [k] \}$ where $Z_1,\dots,Z_k$ are i.i.d. uniform elements of $G$. For any $\ep > 0$, and $t \geq (1+\ep) t_0(k,G)$, with high probability we have
$$
\| \P_{\mathcal{S}}(X(t)= \cdot) - \pi \|_{\mathrm{TV}} =  o(1),
$$
where $\P_{\mathcal{S}}$ denotes the law of $X$ started at $X(0) = \id$ with generator set ${\mathcal{S}}$.
\end{prop}
\begin{proof}
As total variation distance is non-increasing, it suffices to consider $t= (1+\ep) t_0(k,G)$. For simplicity of notation we will drop the time index.
Combining Lemmas~\ref{nonequalC} and~\ref{equalC} in \eqref{tv_decomp} yields
$$|G|\cdot \P(X=X'| \typ)-1=o(1),$$
i.e.,
 \beq\label{expected_L2_distance}
 \E\left[|G|\cdot \P_\mathcal{S}(X=X'|\typ)-1\right]=o(1),
 \eeq
 where the expectation was taken over the choice of $\mathcal{S}$. As 
 $$|G|\cdot \P_\mathcal{S}(X=X'|\typ)-1\geq 4\| \P_\mathcal{S}(X \in \cdot| \C\in\WW^+,\typ^\SS(X)) - \pi \|_{\mathrm{TV}}^2\geq 0$$ 
 almost surely, \eqref{expected_L2_distance} implies that with high probability over the choice of $\mathcal{S}$ we have 
 $$|G|\cdot \P_\mathcal{S}(X=X'|\typ)-1=o(1),$$
and hence with high probability
$$4\| \P_\mathcal{S}(X = \cdot \mid\C\in\WW^+,\typ^\SS(X)) - \pi \|_{\mathrm{TV}}^2=o(1).$$
Since $\P(\C \in \WW^+, \typ^\SS(X)) = 1 - o(1)$ by Lemma~\ref{typical_event_prob}, 
it follows from \eqref{tv_conditional_ub} that, with high probability,
$$\| \P_\mathcal{S}(X =\cdot) - \pi \|_{\mathrm{TV}} \leq \| \P_\mathcal{S}(X = \cdot \mid \C\in\WW^+,\typ^\SS(X) ) - \pi \|_{\mathrm{TV}}+o(1),$$
and the proof is complete.
\end{proof}
\begin{proof}[Proof of Theorem  \ref{cutoff_iid} in Regime $k\lesssim \log|G|$.]
We first reveal $k_\SS$ and $k_\RR$. For any $\eta>0$, with high probability both lie in 
$[(1-\eta)k/2, (1+\eta)k/2]$, so that Assumption~\ref{assump_k_S} holds. 
We condition on these values and treat them as constants close to $k/2$, noting that their precise values only matter to ensure $k_\SS \asymp k$ and $k_\RR \asymp k$ with high probability.

Combining Proposition  \ref{lb_mixing} and \ref{ub_mixing} shows that, for any $\ep>0$, with high probability over the choice of $\mathcal{S}$,
\begin{align*}
\| \P_\mathcal{S}(X(t) = \cdot) - \pi \|_{\mathrm{TV}} =1-o(1)  &\quad\quad\text{ for }t\leq (1-\ep)t_0(k,G),\\
\| \P_\mathcal{S}(X(t) = \cdot) - \pi \|_{\mathrm{TV}} =o(1)  &\quad\quad\text{ for }t\geq (1+\ep)t_0(k,G).
\end{align*}
This establishes the occurrence of cutoff at time $t_0(k,G)$ with high probability.
\end{proof}

\section{Proof in Regime $k\gg\log|G|$}\label{sec:upper_regime}
In this regime, one can establish a more general cutoff result for all virtually Abelian groups, i.e., groups that contain an Abelian subgroup of finite index. The proof builds on and extends arguments from the literature---particularly those in \cite{roichman1996random} and \cite{dou1996enumeration}---to address the continuous-time setting and, more importantly, to establish cutoff throughout the  regime $k \gg \log |G|$.

\mn
\textbf{Theorem \ref{cutoff_virtually_abelian}.}
\textit{
 Let $G=G^{(n)}$ be a finite virtually Abelian group with $n\in\mathbb{N}$. Suppose $H=H^{(n)}$ is a normal Abelian subgroups of $G^{(n)}$ such that $\sup_n |G^{(n)}|/|H^{(n)}|<\infty$. Let $k=k_n \in \mathbb{N}$ and $\mathcal{S}=\{ Z_i^{\pm1}: i\in [k]\}$ with $Z_1,\dots,Z_k$ i.i.d.\ and uniformly distributed over $G$. When $k\gg \log|G|$ and $\log k\ll \log|G|$, the random walk on $\mathrm{Cay}(G,\mathcal{S})$ exhibits cutoff with high probability  at time 
 $$t_0(k,G)= \frac{\log |G|}{\log(k/\log|G|)}.$$ 
}
\begin{remark}
Using the same argument, cutoff can also be shown for a wider class of finite groups satisfying Assumption 1 in \cite{dou1996enumeration}. For simplicity of exposition, we restrict our attention to virtually Abelian groups.
\end{remark}

 Since dihedral groups are virtually Abelian with normal Abelian subgroups $H=\ZZ_n$, Theorem \ref{cutoff_virtually_abelian} yields part of the result in Theorem \ref{cutoff_iid} in the regime $k\gg\log|G|$, hence completing the proof of Theorem \ref{cutoff_iid}. 
Throughout this section, we always assume $t\ll k$ as the proposed mixing time $t_0(k,G)\ll k$. For simplicity of notation, we will often omit the time index $t$ when the dependence on $t$ is clear.

\subsection{Upper bound on mixing time}
Hermon and Olesker-Taylor \cite[\S 7.2]{hermon2021cutoff} adapted Roichman's argument \cite{roichman1996random} to obtain an upper bound on the mixing time that applies to arbitrary finite groups with uniformly random generators. The key idea is to show that there exists a  generator that appears exactly once in the walk; due to the uniform choice of this generator, the resulting walk is then uniformly distributed over the group. This argument turns out to yield the correct upper bound on the mixing times for walks on virtually Abelian groups, where the crucial point is to perform the estimates sharply.

As the proof sketch in \cite[\S 7.2]{hermon2021cutoff} was concise, we present a detailed version here to ensure the discussion is fully self-contained. Throughout this section we will assume $k\gg \log|G|$ and $\log k\ll \log|G|$.
We will sometimes regard $X$ as a sequence, emphasizing the generators applied to obtain $X(t)$ up to a fixed time $t$. In what follows, this time $t$ will be chosen as $(1 \pm \varepsilon)t_0(k,G)$ in the upper and lower bound arguments, respectively. For simplicity, we often omit the time index when it is clear from context.

For each $i \geq 0$, define 
$$
J_i = \{\, a \in [k] : Z_a^{\pm 1} \text{ appears exactly $i$ times in the sequence } X \,\},
$$
and let $J_{\geq 2}:= \bigcup_{i \geq 2} J_i$ for simplicity of notation. Let $X'$ be an independent copy of the walk $X$, and define sets $\{J'_i\}_{i\geq 0}$ for the sequence $X'$ similarly. 

Let $\delta>0$ be a fixed constant and $L=L(\delta):=\min\{\delta t, \delta^2 k/2\}$. Define 
\begin{align}\label{once_typ}
 \typ(\delta)&:=\{ |J_1-te^{-t/k}|\leq \delta t e^{-t/k}, |J_{\geq 2}|\leq L\} \cap \{ |J'_1-te^{-t/k}|\leq \delta t e^{-t/k},|J'_{\geq 2}|\leq L\}.
\end{align}
As each generator arrives according to an independent Poisson process of rate $t/k$, we can see that 
\beq\label{J_i}
|J_i|\sim  \mathrm{Binomial}\left(k,\frac{(t/k)^i}{i!}e^{-t/k}\right).
\eeq
Therefore, standard large deviations calculation implies that for any $\delta>0$,
\beq\label{typical_prob}
\P(\typ(\delta))=1-o(1),
\eeq
as $k\to\infty$, since $t\ll k$.

Recall that for $a\in [k]$, $N_a$ denotes the number of arrivals of $Z_a^{\pm 1}$ in the walk $X$, and $N'_a$ denotes the same for $X'$. Define the event
\beq\label{event_B}
\mathcal{B}=\cup_{a\in [k]}\{ N_a=1, N'_a=0\} \cup\{ N_a=0, N'_a=1\},
\eeq
on which $X(X')^{-1}$ contains a random uniform generator that appears exactly once, so that $X(X')^{-1}$ is uniformly distributed over $G$.

As preparation we first prove the following estimate.
\begin{lemma}\label{largek_ub}
Let $\ep>0$ be fixed and let $t= (1+\ep) \frac{\log |G|}{\log(k/\log|G|)}$. For $\delta>0$ sufficiently small relative to $\ep$, we have $\P(\BB^c|\typ(\delta))=o(|G|^{-1})$.
\end{lemma}
\begin{proof}
We will write $\typ=\typ(\delta)$. Let $\FF=\sigma(J_0,J_1,J_{\geq 2},J'_{\geq 2},\typ)$ denote the $\sigma$-field generated by the information on the indices in $J_0,J_1,J_{\geq 2},J'_{\geq 2}$ and $\1_\typ$. Let $\mathcal{K}=[k]\setminus (J_{\geq 2}\cup J_{\geq 2}')$. On $\BB^c$, as $J_1\cap J'_0=\emptyset$ and $J'_1\cap J_0=\emptyset$, one can deduce that on $\BB^c$,
$$J'_1\cap \mathcal{K}=J_1\cap \mathcal{K}.$$ Further note that conditioning on $\FF$, the indices in $\mathcal{K}$ are exchangeable, i.e., any choice of $|J_1\cap \KK|$ indices from $\KK$ is equally likely to be the indices consisting  $J_1'\cap \KK$. Therefore, letting $K_1=|J_1\cap\KK|$ and $K=|\KK|$,
\begin{align}\label{error_B_complement}
\nonumber\P(\BB^c| \FF)&\leq \P( J'_1\cap \mathcal{K}=J_1\cap \mathcal{K}| \FF)=\frac{ 1}{ { K \choose K_1}}\\
\nonumber&\lesssim \sqrt{ K_1} \exp\left( -K\log K+(K-K_1)\log(K-K_1)+K_1\log K_1\right)\\
&\lesssim \sqrt{K_1} \exp\left(K_1\log(K_1/K)+(K-K_1)\log(1-K_1/K)\right).
\end{align}

Observe that on $\typ$, when $|G|$ is sufficiently large, we have
$$
K_1 \in \left[ (1-\delta) t e^{-t/k} - 2L,\ (1+\delta) t e^{-t/k} \right] \subseteq \left[(1 - 4\delta)t,\ (1 + \delta)t\right] =: [M^-, M^+],
$$
and $K \geq k - 2L \geq (1 - \delta^2)k$. Noting that $K_1 \asymp t \ll k \asymp K$, we can see $|(K-K_1)\log(1-K_1/K)|\asymp K_1\asymp t$. Hence, on $\typ$, we have
\begin{align*}
&K_1\log(K_1/K)+ (K-K_1)\log(1-K_1/K)\\
\leq &M^-\log(M^+/K)+O(t)\leq (1-4\delta)t\log\left(\frac{(1+\delta)t}{(1-\delta^2)k}\right)+O(t)\\
\leq & (1-8\delta)t\log(t/k).
\end{align*}
It then follows from \eqref{error_B_complement} that
\begin{align*}
\P(\BB^c| \FF)&\lesssim  \sqrt{t}\exp\left( (1-8\delta)t\log(t/k) \right).
\end{align*}
Recalling that $t= (1+\ep) \frac{\log |G|}{\log(k/\log|G|)}$, it is not difficult to see 
$$t\log(t/k)=(1+\ep)\log|G|\cdot \left( -1+\frac{\log(1+\ep)-\log\log(k/\log|G|)}{\log(k/\log|G|)}\right)\leq -(1+\ep/2)\log|G|.$$  
Therefore, when $\delta$ is sufficiently small compared to $\ep$, we have
\begin{align*}
\P(\BB^c,\typ)&=\E[ \1_\typ \P(\BB^c|\FF)]\lesssim \sqrt{t}\exp\left( (1-8\delta)t\log(t/k))\right)\\
&\lesssim \sqrt{t}\exp\left( -(1-8\delta)(1+\ep/2)\log|G|\right)=o(|G|^{-1}).
\end{align*}
The proof is completed by recalling $\P(\typ)=1-o(1)$ from \eqref{typical_prob}.
\end{proof}

\mn
\textit{Proof of Theorem \ref{cutoff_virtually_abelian}: Upper Bound on Mixing Times.} Let $\P_{\mathcal{S}}$ denote the law of the random walk given the generator set $\mathcal{S}$.
For any $\ep>0$, we will show that when $t\geq (1+\ep)\frac{\log |G|}{\log(k/\log|G|)}$, with high probability over the choice of $\mathcal{S}$ we have $\| \P_{\mathcal{S}}(X(t)= \cdot)-\pi\|_{TV}=o(1)$. As the total variation distance is non-increasing, it suffices to consider $t=(1+\ep)\frac{\log |G|}{\log(k/\log|G|)}$, and we will drop the time index from now on.

By Lemma \ref{TV_entropic} we have 
$$
\big\| \mathbb{P}_{\mathcal{S}}(X =\cdot) - \pi \big\|_{\mathrm{TV}}
\;\leq\;
\big\| \mathbb{P}_{\mathcal{S}}(X = \cdot \mid \typ(\delta)) - \pi \big\|_{\mathrm{TV}}
\;+\; \mathbb{P}\big( (\typ(\delta))^c \big),
$$
and 
$$4\big\| \mathbb{P}_{\mathcal{S}}(X = \cdot | \typ(\delta)) - \pi \big\|_{\mathrm{TV}}^2\leq |G|\cdot \P_{\mathcal{S}}(X=X'|\typ(\delta))-1.$$

By \eqref{typical_prob}, we have $\P((\typ(\delta)^c)=o(1)$, where the probability is independent of the choice of $\mathcal{S}$. It therefore suffices to show that
$$D(t):=|G|\cdot \P(X=X'|\typ(\delta))-1=\E[ |G|\cdot \P_{\mathcal{S}}(X=X'|\typ(\delta))-1]=o(1),$$
which in turn implies  $\| \P_{\mathcal{S}}(X(t)= \cdot)-\pi\|_{\mathrm{TV}}=o(1)$ with high probability. 

On the event $\mathcal{B}$ defined in \eqref{event_B}, let $Z_a$ denote a generator that appears exactly once in $X(X')^{-1}$ (the sign $\pm 1$ being irrelevant, so we simply write $Z_a$). We can express $X(X')^{-1}=YZ_aY'$ where $Y,Y'$ represent the rest of the sequence and they are independent from $Z_a$. As $Z_a$ is a uniform element of $G$ chosen independently, we can see that $X(X')^{-1}\sim \Unif(G)$ on $\mathcal{B}$. It thus follows from Lemma  \ref{largek_ub} that
$$D(t)\leq |G|\cdot \left(|G|^{-1}+\P(\mathcal{B}^c|\typ(\delta))\right)-1=o(1).$$

\qed

\subsection{Lower bound on mixing time} 

A proof sketch for a lower bound on the mixing time was given by Dou and Hildebrand~\cite{dou1996enumeration} for a certain class of groups under the choice 
$k=\lfloor (\log |G|)^{\alpha}\rfloor$ with $\alpha>1$; see their Assumption~1 and Theorem~4, which also cover the case of virtually Abelian groups. However, in attempting to complete the argument, we found that several key details were omitted, and these are crucial for establishing the proof with full rigor. In addition, a refined combinatorial argument is needed to prove the desired  lower bound for the entire regime  $k \gg \log |G|$. For these reasons, we present the full proofs here.

\medskip

Let $G^{(n)}$ be a virtually Abelian group and $H^{(n)}\trianglelefteq G^{(n)}$ denote its normal Abelian subgroup of  bounded index, i.e., $\sup_n [G^{(n)}:H^{(n)}]=\sup_n |G^{(n)}|/|H^{(n)}|<\infty$. For simplicity of notation, we will drop the index $n$ from now on and write $G=G^{(n)}, H=H^{(n)}$. Let $m=m_n:=[G:H]$ and we can write $G/H=\{b_iH: i\in [m]\}$ and $G=\{ b_i h: i\in [m], h\in H\}$. For each $h\in H$, define its set of conjugates by
$$\mathrm{Conj}(h):=\{ b_i hb_i^{-1}: i\in [m]\}\subseteq H.$$

It is not hard to check that for $h,h'\in H$, either $\mathrm{Conj}(h)\cap \mathrm{Conj}(h')=\emptyset$ or $\mathrm{Conj}(h)=\mathrm{Conj}(h')$. We will refer to $\{\mathrm{Conj}(h):h\in H\}$ as equivalent classes, and say $h,h'\in H$ are in the same equivalent class if $ \mathrm{Conj}(h)=\mathrm{Conj}(h')$. Since each equivalent class has size at most $m$, $H$ contains at least $\lfloor |H|/m \rfloor$ disjoint equivalent classes.

Each generator $Z_a$ can be expressed in the form
\beq\label{BV_decomp}
Z_a=B_aV_a,
\eeq  
where $(B_a)_{a\in [k]}\overset{\mathrm{iid}}{\sim}\Unif (\{ b_i : i\in [m]\})$ and $(V_a)_{a\in[k]}\overset{\mathrm{iid}}{\sim}\Unif(H)$. A set of generators $\mathcal{S}=\{Z_a^{\pm1}: a\in [k]\}$ is said to be \emph{good} if 
\beq
\text{$(V_{a})_{a\in[k]}$ all belong to distinct equivalent classes in $H$}.
\eeq
We first verify that $\mathcal{S}$ is good with high probability, see Lemma \ref{error_good}.
\begin{lemma}\label{error_good}
When $t \ll k$ and $\log k \ll \log|G|$, if $(V_a)_{a\in[k]}\overset{\mathrm{iid}}{\sim}\Unif(H)$ then the resulting generator set $\mathcal{S}=\{Z_a^{\pm1}: a\in [k]\}$ is good with high probability.
\end{lemma}
\begin{proof}
This follows from the simple observation that, since the $(V_a)_{a \in [k]}$ are sampled i.i.d.\ uniformly from $H$, each equivalence class is chosen independently for the sampling of each $a \in [k]$ with probability at most $m/|H|$.
Hence, a simple union bound gives
\begin{align*}
\P( \mathcal{S} \text{ is not good})&\leq |H|\cdot \P( \mathrm{Binomial}( k, m/|H|)\geq 2)\leq |H|\cdot {k\choose 2}(m/|H|)^2 \lesssim k^2/|G|=o(1),
\end{align*}
where the last inequality follows as $ \log k\ll \log|G|$.
\end{proof}

Recall that for $i \ge 1$, 
$
J_{\ge i}= \cup_{\ell \ge i} J_\ell
$
denotes the set of indices of generators that appear at least $i$ times in $X$, and that $N = N(t)$ denotes the total number of steps taken in $X$. Throughout this section, we adopt a new definition of the typical event. Let $\delta>0$ and define
\begin{equation}\label{typical_event_lb}
\typ(\delta) := 
\Bigl\{ \bigl|\, |J_1| - t e^{-t/k} \,\bigr| \le \delta \, t e^{-t/k} \Bigr\} 
\;\cap\; 
\Bigl\{ \bigl|\, N(t) - t \,\bigr| \le \delta t \Bigr\} 
\;\cap\; 
\Bigl\{ \sum_{i \ge 1} |J_{\geq i}| \log(i!) \le t \Bigr\},
\end{equation}
where the inclusion of the last event is a technical point that will become clear in the subsequent argument. 
\begin{lemma}\label{error_typical}
Suppose $t \ll k$ and $\log k \ll \log|G|$. For any $\delta>0$, 
$$
\mathbb{P}\bigl(  \typ(\delta) \bigr) = 1 - o(1).
$$
\end{lemma}

\begin{proof}
The events $ \left\{| |J_1|- t e^{-t/k} |\leq  \delta t e^{-t/k}\right\}, \left\{ |N(t)-t| \leq  \delta t \right\}$ occur with high probability due to standard concentration estimates.  As noted in \eqref{J_i}, each $|J_i|\sim \mathrm{Binomial}(k, \frac{1}{i!}(t/k)^i e^{-t/k})$ and hence
\begin{align*} 
\E\left[|J_{\geq i}|\right]&= \sum_{\ell\geq i}\E[|J_\ell|]= ke^{-t/k}\sum_{\ell \geq i} \frac{(t/k)^\ell }{\ell!}\leq  ke^{-t/k} (t/k)^i\sum_{\ell \geq i} \frac{(t/k)^{\ell-i} }{(\ell-i)!i!}\leq k\frac{(t/k)^i}{i!}.
\end{align*}
Using the fact that $\log(i!)\leq i(i-1)$ for all $i\geq 1$ and $t\ll k$, we obtain
\begin{align*}
\E\left[\sum_{i\geq 1}\log (i!) |J_{\geq i}|\right]&\leq \sum_{i\geq 1} i(i-1)\E\left[|J_{\geq i}|\right]\leq \sum_{i\geq 1}i (i-1) k\frac{(t/k)^i}{i!}=e^{t/k}(t^2/k)\leq 2t^2/k.
\end{align*}
Markov inequality then gives 
$$\P\left( \sum_{i\geq 1}|J_{\geq i}| \log(i!)\geq t\right)\leq 2t/k=o(1).$$

\end{proof}

\begin{lemma}
Suppose $k\gg\log|G|$ and $\log k\ll \log|G|$. For any $\ep>0$ and  $t\leq (1-\ep) \frac{\log |G|}{\log(k/\log|G|)}$, we have, with high probability, 
$$\| \P_\mathcal{S}(X(t)= \cdot)-\pi\|_{\mathrm{TV}}=1-o(1).$$
\end{lemma}

\begin{proof}
As the total variation distance is non-increasing in time, we may, without loss of generality, consider
$
t = (1 - \varepsilon)\, \frac{\log |G|}{\log(k/\log |G|)}.
$
By Lemma \ref{error_good}, it  suffices to prove $\| \P_\mathcal{S}(X(t)= \cdot)-\pi\|_{\mathrm{TV}}=1-o(1)$  for a good generator set $\mathcal{S}$. Thus from now on, we suppose $\mathcal{S}$ is given and is good.

For simplicity of notation, we will drop the time index from now on and write $X=X(t), N=N(t)$. 
Observe that
$$
\left\| \mathbb{P}_{\mathcal{S}}(X= \cdot) - \pi \right\|_{\mathrm{TV}} 
\geq \left\| \mathbb{P}_{\mathcal{S}}(X = \cdot | \typ(\delta) ) - \pi \right\|_{\mathrm{TV}} 
- \mathbb{P}((\typ(\delta))^c).
$$
 Lemma \ref{error_typical} states that $\mathbb{P}((\typ(\delta))^c)=o(1)$ and hence it suffices to show
$
\left\| \mathbb{P}_{\mathcal{S}}(X= \cdot | \typ(\delta)) - \pi \right\|_{\mathrm{TV}} = 1 - o(1).
$

Based on \eqref{BV_decomp} the walk $X$ can be expressed as 
$$X=\prod_{i=1}^{N} (B_{\sigma_i}V_{\sigma_i})=B_{\sigma_1}V_{\sigma_1}\cdots B_{\sigma_{N}}V_{\sigma_{N}}.$$
We will rearrange the sequence $X$ in order to have all $V_{\sigma_i}$'s moved to the right side of the sequence. For instance, the first rearrangement performed is
$$B_{\sigma_{N-1}}V_{\sigma_{N-1}}B_{\sigma_{N}}V_{\sigma_{N}}=B_{\sigma_{N-1}}B_{\sigma_{N}}\bar V_{\sigma_{N-1}}V_{\sigma_{N}},$$
where $\bar V_{\sigma_{N-1}}:=B_{\sigma_{N}}^{-1} V_{\sigma_{N-1}}B_{\sigma_{N}}$ denotes the result of rearrangement. Continuing this way will yield the rearranged sequence
$$X=B_{\sigma_1}\cdots B_{\sigma_{N}} \cdot \bar V_{\sigma_1}\cdots \bar V_{\sigma_{N-1}}\bar V_{\sigma_{N}}$$
where $\bar V_{\sigma_{N}}=V_{\sigma_{N}}$ and $\bar V_{\sigma_i}=(\prod_{j>i} B_{\sigma_j})^{-1}V_{\sigma_i}(\prod_{j>i} B_{\sigma_j})\in\mathrm{Conj}(V_{\sigma_i})$ for all $i< N$.

Since $\mathcal{S}$ is assumed to be good, all $(V_a)_{a \in [k]}$ belong to distinct equivalence classes in $H$. For any $\sigma_i \neq \sigma_j$, we have $\bar V_{\sigma_i} \neq \bar V_{\sigma_j}$. 
For $i \ge 1$, let
\[
|\bar J_i| := \# \Bigl\{\, h \in \bigcup_{a \in [k]} \mathrm{Conj}(V_a) : h \text{ appears exactly $i$ times in } \bar V_{\sigma_1} \cdots \bar V_{\sigma_N} \,\Bigr\}.
\]
Note that  if $h \in \mathrm{Conj}(V_a)$ appears $i$ times in the rearranged sequence $ \bar V_{\sigma_1} \cdots \bar V_{\sigma_N}$, then $V_a$ has to appear at least $i$ times in the original sequence. Consequently, we have
\beq\label{size_bound_barJ_i}
|\bar J_i| \le |J_{\ge i}|.
\eeq

Our goal is to show that conditioned on the event $\typ(\delta)$, the support of $X$ has size $o(|G|)$. Since $\{B_{\sigma_i}\}_{i\leq N}\subseteq \{b_i: i\in [m]\}$, the product $\prod_{i\leq N}B_{\sigma_i}$ has at most $m^{N}$ possible values. We then focus on the support of  
$$\bar V_{\sigma_1}\cdots \bar V_{\sigma_{N-1}}\bar V_{\sigma_{N}}.$$  

Since $\bar V_{\sigma_i} \in \bigcup_{a \in [k]} \mathrm{Conj}(V_a)$, each $\bar V_{\sigma_i}$ can take at most $k m$ possible values. Hence, there can be at most $(km)^N$ different sequences formed by $\bar V_{\sigma_1}\cdots \bar V_{\sigma_{N-1}}\bar V_{\sigma_{N}}$. In addition, because $H$ is Abelian, permuting the elements of any sequence $h_1 h_2 \cdots h_N$ does not change the product, which reduces the number of distinct outcomes. To account for this, observe that the number of distinguished sequences corresponding to the same product is
\[
\frac{N!}{\prod_{i \ge 1} (i!)^{|\bar J_i|}} 
\ge (N/e)^N \exp\Bigl(-\sum_{i \ge 1} |\bar J_i| \log(i!)\Bigr).
\]
On the event $\typ(\delta)$, using \eqref{size_bound_barJ_i}, this quantity is further bounded below by
\[
\frac{N!}{\prod_{i \ge 1} (i!)^{|\bar J_i|}} \ge \exp\bigl( N \log N - N - t \bigr),
\]
which yields  the following upper bound on the size of the support of $X$:
\begin{align}\label{support_size}
\nonumber 
m^N \cdot \frac{(km)^N}{\exp(N\log N - N - t)}
&= \exp\bigl( N\log(k/N) + O(t) \bigr) \le \exp\bigl( (1+\delta)t \log(k/t) + O(t) \bigr) \\
&\le \exp\bigl( (1+2\delta)t \log(k/t) \bigr)
\end{align}
for arbitrarily small $\delta$ when $k$ is sufficiently large.

Finally, it is straightforward to check that 
\[
t\log(k/t) = (1-\ep)\log|G|\left(1+\frac{\log\log(k/\log|G|)-\log(1-\ep)}{\log(k/\log|G|)}\right)
\le (1-\ep/2)\log|G|,
\]
and consequently, 
\[
\eqref{support_size} \;\le\; |G|^{(1+2\delta)(1-\ep/2)} \;\le\; |G|^{\,1-\ep/4}
\]
when $\delta>0$ is chosen sufficiently small relative to $\ep$.  
Therefore, for a given good generator set $\mathcal{S}$, conditioned on $\typ(\delta)$, the support of $X$ has size at most $|G|^{1-\ep/4}$, completing the argument of 
$$\left\| \mathbb{P}_{\mathcal{S}}(X(t) = \cdot \mid \typ(\delta)) - \pi \right\|_{\mathrm{TV}} = 1 - o(1).$$

 \end{proof}

\bibliographystyle{plain}
\bibliography{RW_dihedral_group}

\end{document}